\documentclass[letterpaper,10pt, journal,twoside]{IEEEtran}  

\IEEEoverridecommandlockouts                              

\usepackage{amsmath,amssymb,color,cite}
\usepackage{mathtools}
\usepackage{graphicx}
\usepackage{stfloats}
\usepackage{pmat}
\usepackage{pifont}
\usepackage{latexsym}
\usepackage{verbatim}
\usepackage{cite}
\usepackage[english]{babel}
\usepackage[latin1]{inputenc}
\usepackage{enumerate}
\usepackage{amsmath,amsthm}
\usepackage{tikz}
\usetikzlibrary{tikzmark}
\usepackage{amsfonts}
\usepackage{amssymb}
\usepackage{amsbsy}
\usepackage{bm}
\newcommand{\norm}[1]{\left\lVert#1\right\rVert}

    \let\leq\leqslant
    \let\geq\geqslant

\newcommand{\bmu}{\bm{u}}
\newcommand{\bmx}{\bm{x}}

\newcommand{\bmw}{\bm{w}}

\newcommand{\bmz}{\bm{z}}

\DeclareMathOperator{\rank}{rank}
\DeclareMathOperator{\im}{im}
\DeclareMathOperator{\trace}{tr}

\newtheorem{theorem}{Theorem}
\newtheorem{lemma}[theorem]{Lemma}
\newtheorem{corollary}[theorem]{Corollary}
\theoremstyle{definition}
\newtheorem{definition}[theorem]{Definition}

\newtheorem{problem}{Problem}

\newtheorem{remark}[theorem]{Remark}
\newtheorem{assumption}{Assumption}

\title{From noisy data to feedback controllers: non-conservative design via a matrix S-lemma}

\author{Henk J. van Waarde, M. Kanat Camlibel, and Mehran Mesbahi%
	\thanks{}
	\thanks{Henk van Waarde and Kanat Camlibel are with the Bernoulli Institute for Mathematics, Computer Science, and Artificial Intelligence, University of Groningen, Nij\-enborgh 9, 9747 AG, Groningen, The Netherlands. Henk van Waarde is also with the Engineering and Technology Institute Groningen, University of Groningen, Nij\-enborgh 4, 9747 AG, Groningen, The Netherlands. Mehran Mesbahi is with the William E. Boeing Department of Aeronautics and Astronautics, University of Washington, Seattle, WA 98195 USA (email: {\footnotesize{\tt h.j.van.waarde@rug.nl; m.k.camlibel@rug.nl;mesbahi@uw.edu}}).}
}

\begin{document}

\maketitle
\begin{abstract}
We propose a new method to obtain feedback controllers of an unknown dynamical system directly from noisy input/state data. The key ingredient of our design is a new matrix S-lemma that will be proven in this paper. We provide both strict and non-strict versions of this S-lemma, that are of interest in their own right. Thereafter, we will apply these results to data-driven control. In particular, we will derive non-conservative design methods for quadratic stabilization, $\mathcal{H}_2$ and $\mathcal{H}_{\infty}$ control, all in terms of data-based linear matrix inequalities. In contrast to previous work, the dimensions of our decision variables are independent of the time horizon of the experiment. Our approach thus enables control design from large data sets. 
\end{abstract}

\section{Introduction}
\IEEEPARstart{I}{n} this paper we study the problem of designing control laws for an unknown dynamical system using noisy data. This general problem exists for a long time, but has seen a renewed surge of interest over the last few years. The problem can be approached via different angles, for example using combined system identification and model-based control, or by computing control laws from data without the intermediate modeling step. We will contribute to the second category of methods, aiming at control design directly from noisy data.

One of the main challenges in this area is to come up with robust control laws that guarantee stability and performance of the unknown system despite the inherent uncertainty caused by noisy data. Even though there are several recent contributions addressing this issue, there are multiple open questions. In fact, one of the unsolved problems is to come up with \emph{non-conservative} control design strategies using only a finite number of data samples. 

We will tackle this problem by providing necessary and sufficient conditions on noisy data under which controllers can be obtained. As a consequence, our ensuing control technique is non-conservative, and also shown to be tractable from a computational point of view. The technical ingredient that enables our design is a new generalization of the classical S-lemma \cite{Yakubovich1977,Polik2007}, which will be proven in this paper. We will formulate our control problems using the general data informativity framework as introduced in \cite{vanWaarde2020}. As such, the results developed in this paper can be seen as a natural extension of those in \cite{vanWaarde2020} to noisy data. 

\subsection*{Literature on data-driven control}
The literature on data-driven control is expanding rapidly. Our account of previous work is therefore not exhaustive, but we note that additional references can be found in the survey \cite{Hou2013}. We mention contributions to data-driven optimal control \cite{Skelton1994,Furuta1995,Shi1998,Favoreel1999b,Aangenent2005,Pang2018,daSilva2019,Baggio2019,Mukherjee2018,Alemzadeh2019,Dean2019,vanWaarde2020c}, PID control \cite{Keel2008,Fliess2013}, predictive control \cite{Favoreel1999,Salvador2018,Huang2019,Alpago2020,Hewing2020}, and nonlinear control \cite{Tabuada2017,Dai2018,Tabuada2020,Bisoffi2020,Guo2020}. Some of these techniques are iterative in nature: the controller is updated online when new data are presented. Examples of this include policy iteration methods \cite{Bradtke1993} and iterative feedback tuning \cite{Hjalmarsson1998}. Other methods are one-shot in the sense that the controller is constructed offline from a batch of data. We mention, for instance, virtual reference feedback tuning \cite{Campi2002} and methods based on Willems' fundamental lemma \cite{Willems2005} (see also \cite{vanWaarde2020b}). The latter line of work has been quite fruitful, with contributions ranging from output matching \cite{Markovsky2008} and control by interconnection \cite{Maupong2017}, to data-enabled predictive control \cite{Coulson2019} and a data-based closed-loop system parameterization \cite{DePersis2020}. This parameterization has been used for stabilizing and optimal control design using data-based linear matrix inequalities \cite{DePersis2020}. We also mention the extension \cite{DePersis2020b} studying LQR using noisy data, and the paper \cite{Berberich2019c} for a closed-loop parameterization using noisy data. Additional recent research directions include data-driven control of networks \cite{Allibhoy2020,Baggio2020} and the interplay between data-guided control and model reduction \cite{Monshizadeh2020}. 

\subsection*{Review of the S-lemma}
First proven by Yakubovich \cite{Yakubovich1977}, the \emph{S-lemma} is a classical result in control theory and optimization \cite{Polik2007}. The result revolves around the question when the non-negativity of one quadratic function implies that of another. The crux of the S-lemma is that this seemingly difficult implication is equivalent to the feasibility of a linear matrix inequality (LMI) in a scalar variable, called a \emph{multiplier}. The act of replacing the implication by a linear matrix inequality is often referred to as the \emph{S-procedure}. The S-procedure is more generally applicable in situations where one quadratic inequality is implied by \emph{multiple} quadratic inequalities \cite{Boyd1994}. In this case, multiple scalar multipliers are used. It is well-known, however, that the S-procedure is conservative in general, although there exist special cases in which losslessness (a là S-lemma) can be shown \cite[Sec.~3]{Polik2007}. A generalized S-lemma involving more general types of multipliers was considered in \cite{Iwasaki2000,Iwasaki2005}, and was shown to be non-conservative under extra assumptions. 

The classical S-lemma, as well as the contributions mentioned above, all deal with vector variables. 
However, for reasons that will become clear in Section \ref{sectionproblem}, we need a type of S-lemma that is applicable to quadratic functions of \emph{matrix} variables. Such a result has been reported for specific quadratic functions in \cite[Thm. 3.3]{Luo2004}. In the special case that variables are \emph{bounded}, an S-lemma for matrix variables can also be derived by combining the so-called \emph{full block S-procedure} \cite{Scherer1997,Iwasaki1997,Scherer2001,Iwasaki2001} with results from the literature on LMI relaxations \cite{Scherer2005,Scherer2006,Scherer2006b}. These specific results are, however, less suited for the application of data-driven control that we have in mind. Therefore, in this paper we derive general matrix S-lemmas for both strict and non-strict inequalities. Our matrix S-lemma for non-strict inequalities is a direct generalization of the classical S-lemma \cite{Yakubovich1977}. It also recovers the result from \cite{Luo2004} as a special case. As a corollary of our matrix S-lemma for a strict inequality, we recover the S-lemma derived from the general theory on LMI relaxations \cite{Scherer2005}.

\subsection*{Our contributions}

The core of our approach is to formulate data-driven control as the problem of deciding whether one quadratic matrix inequality is implied by another one. Our first contribution is to extend the classical S-lemma to quadratic matrix inequalities. Our second contribution is to apply these results to data-driven control. In particular, we come up with design procedures for quadratic stabilization, $\mathcal{H}_2$ control and $\mathcal{H}_{\infty}$ control. 

Throughout the paper we will assume no statistics on the noise, but we will work with general bounded disturbances. We are thereby inspired by recent papers \cite{DePersis2020,Berberich2019c} that formalize the assumption of bounded disturbances in terms of quadratic matrix inequalities. In fact, we will work with an assumption on the noise that is closely related to that of \cite{Berberich2019c}, and is more general than the assumption in \cite{DePersis2020}. In terms of control design, our approach completely differs from the above papers. In fact, instead of working with data-based parameterizations of closed-loop systems \cite{DePersis2020,Berberich2019c,DePersis2020b}, we will work with a representation of all \emph{open-loop} systems explaining the data, akin to the framework of \cite{vanWaarde2020}. We believe that our approach is attractive for the following reasons:
\begin{enumerate}
    \item We provide robust guarantees on the stability and performance of the unknown data-generating system. The design involves data-guided LMI's that are tractable from a computational point of view and are easy to implement. 
    \item By virtue of our matrix S-lemma, the design method is \emph{non-conservative}. This is in contrast with previous LMI formulations in \cite{DePersis2020,Berberich2019c} that provide sufficient conditions for controller design.
    \item Last but not least, the variables involved in our method are independent of the time horizon of the experiment. Our approach is thus applicable to large data sets. This is an advantage over closed-loop system parameterizations \cite{DePersis2020,Berberich2019c}, that become computationally intractable when applied to big data.
\end{enumerate}

\subsection*{Outline of the paper}
In Section \ref{sectionproblem} we will formulate the problem. Section \ref{sectionSlemma} contains our results on the matrix S-lemma. These results are then applied to data-driven stabilization in Section \ref{sectionddstab}, and to data-driven $\mathcal{H}_2$ control and $\mathcal{H}_{\infty}$ control in Section \ref{sectionperformance}. In Section \ref{sectionsimulations} we provide simulation examples. Finally, our conclusions are provided in Section \ref{sectionconclusions}.

\section{The problem of data-driven stabilization}
\label{sectionproblem}
Consider the linear time-invariant system
\begin{equation}
    \label{system1}
    \bmx(t+1) = A_s \bmx(t) + B_s \bmu(t) + \bmw(t),
\end{equation}
where $\bmx \in \mathbb{R}^n$ denotes the state, $\bmu \in \mathbb{R}^m$ is the input and $\bmw \in \mathbb{R}^n$ is an unknown noise term. The matrices $A_s \in \mathbb{R}^{n \times n}$ and $B_s \in \mathbb{R}^{n \times m}$ denote the unknown state and input matrices. Our goal is to design stabilizing controllers for \eqref{system1} on the basis of a finite number of measurements of the state and input of the system. To this end, suppose that we measure state and input data on a time interval\footnote{All our results are still true for data collected on multiple intervals, see \cite[Ex. 2]{vanWaarde2020} for more details on how to arrange the data matrices in this case.}, and collect these samples in the matrices
\begin{subequations}
\begin{align*}
X &:= \begin{bmatrix} x(0) & x(1) & \cdots & x(T)\end{bmatrix}, \\
U_- &:= \begin{bmatrix} u(0) & u(1) & \cdots & u(T-1)\end{bmatrix}.
\end{align*}
\end{subequations}
By defining the matrices
\begin{subequations}
\begin{align*}
X_+& := \begin{bmatrix} x(1) & x(2) & \cdots & x(T) \end{bmatrix}, \\
X_-& := \begin{bmatrix} x(0) & x(1) & \cdots & x(T-1) \end{bmatrix}, \\
W_-& := \begin{bmatrix} w(0) & w(1) & \cdots & w(T-1) \end{bmatrix}, 
\end{align*}	
\end{subequations}
we clearly have 
\begin{equation}
\label{eqdatanoise}
    X_+ = A_s X_- + B_s U_- + W_-.
\end{equation}
We emphasize that the system matrices $A_s$ and $B_s$ as well as the noise term $W_-$ are \emph{unknown}, while $X$ and $U_-$ are measured. Before we introduce the problem we will explain our assumption on the noise $W_-$. 

\subsection{Assumption on the noise}
We will formalize our assumption on the noise in terms of a quadratic matrix inequality. 
\begin{assumption}
\label{assumption1}
The noise samples $w(0),w(1),\dots,w(T-1)$, collected in the matrix $W_-$, satisfy the bound
\begin{equation}
    \label{asnoise}
    \begin{bmatrix}
    I \\ W_-^\top 
    \end{bmatrix}^\top 
    \begin{bmatrix}
    \Phi_{11} & \Phi_{12} \\
    \Phi_{12}^\top & \Phi_{22}
    \end{bmatrix}
    \begin{bmatrix}
    I \\ W_-^\top 
    \end{bmatrix} \geq 0,
\end{equation}
for known matrices $\Phi_{11} = \Phi_{11}^\top$, $\Phi_{12}$ and $\Phi_{22} = \Phi_{22}^\top < 0$. 
\end{assumption}
Note that the negative definiteness of $\Phi_{22}$ ensures that the set of noise matrices $W_-$ satisfying \eqref{asnoise} is bounded. In the special case $\Phi_{12} = 0$ and $\Phi_{22} = -I$, \eqref{asnoise} reduces to 
\begin{equation}
    \label{redasnoise}
    W_- W_-^\top = \sum_{t=0}^{T-1} w(t) w(t)^\top \leq \Phi_{11}.
\end{equation}
The inequality \eqref{redasnoise} has the interpretation that the energy of $\bmw$ is bounded on the finite time interval $[0,T-1]$. If $\bmw$ is a random variable, its \emph{sample covariance matrix} is given by 
$$
\frac{1}{T-1} W_- (I-\frac{1}{T}J)W_-^\top,
$$
where $J$ is the matrix of ones. Thus, \eqref{asnoise} can also capture known bounds on the sample covariance by the choices $\Phi_{12} = 0$ and $\Phi_{22} = -\frac{1}{T-1} (I-\frac{1}{T}J)$. We emphasize, however, that we do not make any assumptions on the statistics of $\bmw$ and work with the general bound \eqref{asnoise} instead. Note that \cite[Asm. 5]{DePersis2020} is a special case of Assumption \ref{assumption1} for the choices $\Phi_{11} = \gamma X_+ X_+^\top$ with $\gamma > 0$, $\Phi_{12} = 0$ and $\Phi_{22} = -I$. We remark that norm bounds on the individual noise samples $w(t)$ also give rise to bounds of the form \eqref{asnoise}, although this may lead to some conservatism. Indeed,  note that $\norm{w(t)}_2^2 \leq \epsilon$ implies that $w(t) w(t)^\top \leq \epsilon I$ for all $t$. As such, the bound \eqref{redasnoise} is satisfied for $\Phi_{11} = T \epsilon I$.
\begin{remark}
	\label{remarkBerb}
Note that the noise model in Assumption \ref{assumption1} is the ``transposed" of the model in \cite{Berberich2019c}, in the sense that we penalize, e.g., the term $W_- \Phi_{22} W_-^\top$ instead of a term $W_-^\top Q_w W_-$. In some cases, these two different noise models are actually equivalent. For example, if $\Phi_{11} > 0$ and $\Phi_{12} = 0$ then \eqref{asnoise} can be written via a Schur complement argument as
$$
\begin{bmatrix}
\Phi_{11} & W_- \\ W_-^\top & -\Phi_{22}^{-1}
\end{bmatrix} \geq 0.
$$
In turn, this is equivalent to $-\Phi_{22}^{-1} - W_-^\top \Phi_{11}^{-1} W_- \geq 0$, which is of the same form as \cite{Berberich2019c}.
\end{remark}

\begin{remark}
\label{remarkE}
In some cases, we may know a priori that the noise $\bmw$ does not directly affect the entire state-space, but is contained in a subspace. This prior knowledge can be captured by the noise model in Assumption \ref{assumption1}. Indeed, $W_-$ is of the form $W_- = E \hat{W}_-$ for some $\hat{W}_- \in \mathbb{R}^{r\times T}$ satisfying 
$$
\begin{bmatrix}
    I \\ \hat{W}_-^\top 
    \end{bmatrix}^\top 
    \begin{bmatrix}
    \hat{\Phi}_{11} & \hat{\Phi}_{12} \\
    \hat{\Phi}_{12}^\top & \hat{\Phi}_{22}
    \end{bmatrix}
    \begin{bmatrix}
    I \\ \hat{W}_-^\top 
    \end{bmatrix} \geq 0
$$
if and only if 
$$
\begin{bmatrix}
    I \\ W_-^\top 
    \end{bmatrix}^\top 
    \begin{bmatrix}
    E \hat{\Phi}_{11} E^\top & E \hat{\Phi}_{12} \\
    \hat{\Phi}_{12}^\top E^\top & \hat{\Phi}_{22}
    \end{bmatrix}
    \begin{bmatrix}
    I \\ W_-^\top 
    \end{bmatrix} \geq 0.
$$
Thus, the conclusion is that we can incorporate the knowledge that $W_- \in \im E$ by appropriate choices of the $\Phi$-matrices in \eqref{asnoise}. Showing the above claim is straightforward: note that the ``only if" statement follows by pre- and post-multiplication with $E$ and $E^\top$, respectively. The ``if" part follows by noting that $x \in \ker E^\top$ implies $x^\top W_- \hat{\Phi}_{22} W_-^\top x \geq 0$, thus $W_-^\top x = 0$. Hence, $\ker E^\top \subseteq \ker W_-^\top$, equivalently, $\im W_- \subseteq \im E$.
\end{remark}

\subsection{Problem formulation}

We will follow the general framework for data-driven analy\-sis and control in \cite{vanWaarde2020}. To this end, we define the set of all systems $(A,B)$ explaining the data $(U_-,X)$, i.e., all $(A,B)$ satisfying 
\begin{equation}
    \label{dataeq}
    X_+ = A X_- + B U_- + W_-
\end{equation}
for some $W_-$ satisfying \eqref{asnoise}. We denote this set by $\Sigma$: 
$$
\Sigma := \{ (A,B) \mid \eqref{dataeq} \text{ holds for some } W_- \text{ satisfying } \eqref{asnoise} \}.
$$
We can only guarantee that a state feedback $\bmu = K \bmx$ stabilizes the true system $(A_s,B_s)$ if it stabilizes \emph{all} systems in $\Sigma$. This motivates the following definition of \emph{informative} data. Loosely speaking, data are called informative if they enable the design of a controller that stabilizes all systems in $\Sigma$ (and thus, the unknown $(A_s,B_s)$).

\begin{definition}
\label{definformativity}
Assume that $(U_-,X)$ satisfies \eqref{eqdatanoise} for some $W_-$ satisfying Assumption \ref{assumption1}.
The data $(U_-,X)$ are called \emph{informative for quadratic stabilization} if there exists a feedback gain $K$ and a matrix $P = P^\top > 0$ such that
\begin{equation}
    \label{lyapunovineq}
P - (A+BK) P (A+BK)^\top > 0
\end{equation}
for all $(A,B) \in \Sigma$.
\end{definition}

Note that in particular, we are interested in \emph{quadratic stabilization} and we ask for a \emph{common} Lyapunov matrix $P$ for all $(A,B) \in \Sigma$. We will not treat $(A,B)$-dependent Lyapunov matrices in this paper, but consider this case for future work instead.  

Definition~\ref{definformativity} leads to two natural problems. First, we are interested in the question under which conditions the data are informative. We formalize this in the following problem.

\begin{problem}[Informativity]
\label{problem1}
Assume that $(U_-,X)$ satisfies \eqref{eqdatanoise} for some $W_-$ satisfying Assumption \ref{assumption1}.
Find necessary and sufficient conditions under which the data $(U_-,X)$ are informative for quadratic stabilization. 
\end{problem}

The second problem is a design issue: we are interested in procedures to come up with a feedback that stabilizes all systems in $\Sigma$.

\begin{problem}[Control design]
\label{problem2}
Assume that $(U_-,X)$ satisfies \eqref{eqdatanoise} for some $W_-$ satisfying Assumption \ref{assumption1}.
If the data $(U_-,X)$ are informative for quadratic stabilization, find a stabilizing feedback gain $K$ such that \eqref{lyapunovineq} is satisfied for all $(A,B) \in \Sigma$.
\end{problem}
In addition to data-driven stabilization, we are also interested in including performance specifications. Natural extensions to Problems \ref{problem1} and \ref{problem2} will be discussed in Section~\ref{sectionperformance}.

\subsection{Our approach}
In what follows, we will outline our strategy for solving Problems \ref{problem1} and \ref{problem2}. Let $(A,B) \in \Sigma$ and rewrite \eqref{dataeq} as
\begin{equation}
    \label{Wrelation}
    W_- = X_+ - A X_- - B U_-.
\end{equation}
Recall that by Assumption \ref{assumption1}, we have 
$$
\begin{bmatrix}
    I \\ W_-^\top 
    \end{bmatrix}^\top 
    \begin{bmatrix}
    \Phi_{11} & \Phi_{12} \\
    \Phi_{12}^\top & \Phi_{22}
    \end{bmatrix}
    \begin{bmatrix}
    I \\ W_-^\top 
    \end{bmatrix} \geq 0.
$$
By substitution of \eqref{Wrelation}, this yields
\begin{equation}
    \label{ineqAB}
    \begin{bmatrix}
    I \\ A^\top \\ B^\top 
    \end{bmatrix}^\top 
    \begin{bmatrix}
    I & X_+ \\ 0 & -X_- \\ 0 & -U_-
    \end{bmatrix}
    \begin{bmatrix}
    \Phi_{11} & \Phi_{12} \\
    \Phi_{12}^\top & \Phi_{22}
    \end{bmatrix}
    \begin{bmatrix}
    I & X_+ \\ 0 & -X_- \\ 0 & -U_-
    \end{bmatrix}^\top
    \begin{bmatrix}
    I \\ A^\top \\ B^\top 
    \end{bmatrix} \geq 0.
\end{equation}
This shows that $A$ and $B$ satisfy a \emph{quadratic matrix inequality} (QMI) of the form \eqref{ineqAB}\footnote{We note that quadratic uncertainty descriptions have also arisen in the papers \cite{Umenberger2019,Ferizbegovic2020,Iannelli2020} studying data-driven control under the assumption that $\bmw$ is a normally distributed process noise.}. In fact, the set $\Sigma$ of all systems explaining the data can be equivalently characterized in terms of \eqref{ineqAB}, as asserted in the following lemma.

\begin{lemma}
\label{lemmaSigma}
We have that $\Sigma = \left\{ (A,B) \mid \eqref{ineqAB} \text{ is satisfied} \right\}$.
\end{lemma}

\begin{proof}
Suppose that $(A,B) \in \Sigma$. Then \eqref{Wrelation} is satisfied for some $W_-$ satisfying \eqref{asnoise}. This means that \eqref{ineqAB} holds. As such
$$
\Sigma \subseteq \left\{ (A,B) \mid \eqref{ineqAB} \text{ is satisfied} \right\}.
$$
To prove the reverse inclusion, let $(A,B)$ be such that \eqref{ineqAB} is satisfied. Define $W_- := X_+ - A X_- - BU_-$. By \eqref{ineqAB}, $W_-$ satisfies the assumption \eqref{asnoise}. Since \eqref{dataeq} holds for $(A,B)$ by construction, we conclude that $(A,B) \in \Sigma$. 
\end{proof}
By Lemma~\ref{lemmaSigma} the set $\Sigma$ of systems explaining the data is characterized by a quadratic matrix inequality in $(A,B)$. Next, we turn our attention to the design condition \eqref{lyapunovineq}. Suppose that we fix\footnote{We make this hypothesis purely to explain the ideas behind our approach. In fact, in Section~\ref{sectionddstab} we show how $P$ and $K$ can be computed from data.} a Lyapunov matrix $P = P^\top > 0$ and a feedback gain $K$. Note that the inequality \eqref{lyapunovineq} is equivalent to 
\begin{equation}
\label{ineqABPK}
    \begin{bmatrix}
    I \\ A^\top \\ B^\top 
    \end{bmatrix}^\top
    \begin{bmatrix}
    P & 0 & 0 \\
    0 & -P & -PK^\top \\
    0 & -KP & -KPK^\top
    \end{bmatrix}
    \begin{bmatrix}
    I \\ A^\top \\ B^\top 
    \end{bmatrix} > 0,
\end{equation}
which is yet another quadratic matrix inequality in $A$ and $B$. Therefore, Problem \ref{problem1} essentially boils down to understanding under which conditions the quadratic matrix inequality \eqref{ineqABPK} holds for all $(A,B)$ satisfying the quadratic matrix inequality \eqref{ineqAB}. Data-driven stabilization thus naturally leads to the following fundamental question:  
$$
\textit{When does one QMI imply another QMI?}
$$
The familiar reader will immediately recognize the similarity between the above question and the statement of the so-called \emph{S-lemma} \cite{Polik2007}. In fact, the S-lemma provides conditions under which the non-negativity of one quadratic function implies that of another one. This motivates the following section, in which we generalize the S-lemma to matrix variables.

\section{The matrix-valued S-lemma}
\label{sectionSlemma}
In this section we present a new S-lemma with matrix variables. Before we do so, we provide a brief recap on the classical S-lemma. 
\subsection{Recap of the classical S-lemma}
A function $f: \mathbb{R}^n \to \mathbb{R}$ is called \emph{quadratic} if it can be written in the form
\begin{align}
f(x) &= \begin{bmatrix}
     1 \\ x
     \end{bmatrix}^\top \begin{bmatrix}
     M_{11} & M_{12} \\ M_{12}^\top & M_{22}
     \end{bmatrix} \begin{bmatrix}
     1 \\ x
     \end{bmatrix}, \label{f(x)} 
\end{align}
for some $M_{11} \in \mathbb{R}$, $M_{12} \in \mathbb{R}^{1 \times n}$ and $M_{22} = M_{22}^\top \in \mathbb{R}^{n \times n}$. A homogeneous quadratic function of the form $f(x) = x^\top M_{22} x$ is called a \emph{quadratic form}. The following theorem describes the celebrated S-lemma, proven by Yakubovich in \cite{Yakubovich1977}, see also \cite[Thm. 2.2]{Polik2007}.

\begin{theorem}[S-lemma]
\label{Slemma}
Let $f,g : \mathbb{R}^n \to \mathbb{R}$ be quadratic functions. Suppose that there exists $\bar{x} \in \mathbb{R}^n$ such that $g(\bar{x}) > 0$. Then $f(x) \geq 0$ for all $x \in \mathbb{R}^n$ such that $g(x) \geq 0$ if and only if there exists a scalar $\alpha \geq 0$ such that
\begin{equation}
    \label{fg}
f(x) - \alpha g(x) \geq 0 \quad \forall x \in \mathbb{R}^n.
\end{equation}
\end{theorem}
We note that the functions $f$ and $g$ are not assumed to be convex. As such, it appears to be difficult to check the condition $f(x) \geq 0$ for all $x \in \mathbb{R}^n$ satisfying $g(x) \geq 0$. The importance of the S-lemma lies in the fact that the characterization \eqref{fg} of this condition is equivalent to a \emph{linear matrix inequality}
\begin{equation*}
    \begin{bmatrix}
     M_{11} & M_{12} \\ M_{12}^\top & M_{22}
     \end{bmatrix} - \alpha \begin{bmatrix}
     N_{11} & N_{12} \\ N_{12}^\top & N_{22}
     \end{bmatrix} \geq 0
\end{equation*}
in the scalar variable $\alpha \geq 0$. Here the matrices $N_{11} \in \mathbb{R}$, $N_{12} \in \mathbb{R}^{1 \times n}$ and $N_{22} \in \mathbb{R}^{n \times n}$ define the quadratic function $g$ analogous to \eqref{f(x)}. 

The scalar $\alpha$ is called a \emph{multiplier} and the assumption $g(\bar{x}) > 0$ for some $\bar{x} \in \mathbb{R}^n$ is often referred to as the \emph{Slater condition}. This assumption is necessary in the sense that Theorem~\ref{Slemma} is false without it. To show this by means of an example, one can take, e.g., $f(x) = x^\top A x$ and $g(x) = - x^\top B x$ with $A$ and $B$ as in the example of \cite[Page 4476]{Zi-zong2010}. A version of the S-lemma where $g$ satisfies a strict inequality has been presented in \cite[Thm. 7.8]{Polik2007}. We will reformulate the result in the following theorem. 

\begin{theorem}[Strict S-lemma]
\label{sSlemma}
Let $f,g : \mathbb{R}^n \to \mathbb{R}$ be quadratic forms. Suppose that there exists an $\bar{x} \in \mathbb{R}^n$ such that $g(\bar{x}) > 0$. Then $f(x) \geq 0$ for all $x \in \mathbb{R}^n$ such that $g(x) > 0$ if and only if there exists a scalar $\alpha \geq 0$ such that 
$$
f(x) - \alpha g(x) \geq 0 \quad \forall x \in \mathbb{R}^n.
$$
\end{theorem}
Note that Theorem~\ref{sSlemma} is stated with two multipliers in \cite[Thm. 7.8]{Polik2007}. However, the inclusion of the Slater condition allows us to state Theorem~\ref{sSlemma} with a single multiplier $\alpha$. 

\subsection{S-lemma with matrix variables}
Next, we aim at generalizing Theorems \ref{Slemma} and \ref{sSlemma} to quadratic functions of the form
\begin{equation}
\label{partitionMN}
\begin{bmatrix}
I \\ X
\end{bmatrix}^\top \begin{bmatrix}
M_{11} & M_{12} \\ M_{12}^\top & M_{22}
\end{bmatrix} 
\begin{bmatrix}
I \\ X
\end{bmatrix} \text{and } \begin{bmatrix}
I \\ X
\end{bmatrix}^\top \begin{bmatrix}
N_{11} & N_{12} \\ N_{12}^\top & N_{22}
\end{bmatrix} 
\begin{bmatrix}
I \\ X
\end{bmatrix},
\end{equation}
where $X \in \mathbb{R}^{n \times k}$ is a \emph{matrix variable}, and the partitioned matrices $M,N \in \mathbb{R}^{(k+n) \times (k+n)}$ are real and symmetric. As our first step, the following theorem provides an S-lemma for homogeneous quadratic functions of the form $X^\top M X$ and $X^\top N X$. Naturally, instead of the non-negativity of functions in the classical S-lemma, we now consider the positive (semi)definiteness of quadratic functions of matrix variables. 

\begin{theorem}[Homogeneous matrix S-lemma]
\label{hommatSlemma}
Let $M,N \in \mathbb{R}^{n \times n}$ be symmetric matrices and assume that $\bar{X}^\top N \bar{X} > 0$ for some $\bar{X} \in \mathbb{R}^{n \times k}$. The following statements are equivalent:
\begin{enumerate}[(i)]
    \item $X^\top M X \geq 0$ for all $X \in \mathbb{R}^{n \times k}$ such that $X^\top N X \geq 0$.\label{homi}
    \item $X^\top M X \geq 0$ for all $X \in \mathbb{R}^{n \times k}$ such that $X^\top N X > 0$.\label{homii}
    \item There exists a scalar $\alpha \geq 0$ such that $M - \alpha N \geq 0$.\label{homiii}
\end{enumerate}
\end{theorem}

\begin{remark}
The assumption on the existence of $\bar{X}$ such that $\bar{X}^\top N \bar{X} > 0$ is a natural generalization of the Slater condition in Theorems \ref{Slemma} and \ref{sSlemma}. The assumption is again necessary in the sense that Theorem~\ref{hommatSlemma} is false without it. Nonetheless, it can be shown that the assumption can be weakened if one is interested only in the equivalence of \eqref{homi} and \eqref{homiii}. In fact, one can show using similar arguments as in the proof of Theorem~\ref{hommatSlemma} that \eqref{homi} $\iff$ \eqref{homiii} under the assumption that $\exists \bar{x} \in \mathbb{R}^n$ such that $\bar{x}^\top N \bar{x} > 0$, i.e., under the ``standard" Slater condition. 
\end{remark}

\begin{proof}[Proof of Theorem~\ref{hommatSlemma}]
It is clear that (\ref{homi}) $\implies$ (\ref{homii}) and (\ref{homiii}) $\implies$ (\ref{homi}). As such, it suffices to prove the implication (\ref{homii}) $\implies$ (\ref{homiii}). To this end, suppose that (\ref{homii}) holds. Let $x \in \mathbb{R}^n$ be such that $x^\top N x > 0$. We want to prove that $x^\top M x \geq 0$ so that we can apply Theorem~\ref{sSlemma}. Choose a vector $v \in \mathbb{R}^k$ such that $\norm{v} = 1$. Next, we define the matrix $X \in \mathbb{R}^{n \times k}$ as $X := \epsilon \bar{X} + x v^\top$ for $\epsilon \neq 0$. Clearly, $X^\top N X$ is equal to
$$
\epsilon^2 \bar{X}^\top N \bar{X} + \epsilon \left(\bar{X}^\top N x v^\top + v x^\top N \bar{X}\right) + (x^\top N x) v v^\top.
$$
We claim that $X^\top N X$ is positive definite for $\epsilon$ sufficiently small. To prove this claim, first suppose that $y \in \mathbb{R}^k$ is nonzero and $v^\top y = 0$. Then we obtain
$$
y^\top X^\top N X y = \epsilon^2 y^\top \bar{X}^\top N \bar{X} y > 0.
$$
Secondly, suppose that $y \in \mathbb{R}^k$ is nonzero and $v^\top y =: \beta \neq 0$. Then $y^\top X^\top N X y$ is equal to 
$$
 y^\top \left( \epsilon^2 \bar{X}^\top N \bar{X} + \epsilon \left(\bar{X}^\top N x v^\top + v x^\top N \bar{X}\right) \right) y + (x^\top N x) \beta^2,
$$
which is positive for $\epsilon$ sufficiently small since $\beta \neq 0$ and $x^\top N x > 0$. We conclude that $X^\top N X >0$ for $\epsilon$ sufficiently small. Now, by (\ref{homii}) we conclude that $X^\top M X \geq 0$.  Multiplication of the latter inequality from left by $v^\top$ and right by $v$ yields the inequality
\begin{equation}
    \label{ineqM}
\epsilon^2 v^\top \bar{X}^\top M \bar{X} v + \epsilon \left( v^\top \bar{X}^\top M x + x^\top M \bar{X} v \right) + x^\top M x \geq 0.
\end{equation}
This implies that $x^\top M x \geq 0$. Indeed, if $x^\top M x < 0$ then there exists a sufficiently small $\epsilon \neq 0$ such that 
$$
\epsilon^2 v^\top \bar{X}^\top M \bar{X} v + \epsilon \left( v^\top \bar{X}^\top M x + x^\top M \bar{X} v \right) + x^\top M x < 0,
$$
which contradicts \eqref{ineqM}. To conclude, we have shown that $x^\top M x \geq 0$ for all $x \in \mathbb{R}^n$ such that $x^\top N x > 0$. By Theorem~\ref{sSlemma}, the condition (\ref{homiii}) is satisfied. This proves the theorem. 
\end{proof}

Next, we build on Theorem~\ref{hommatSlemma} by introducing a general (inhomogeneous) S-lemma with matrix variables. The following theorem is one of the main results of this section. 

\begin{theorem}[Matrix S-lemma]
\label{matSlemma}
Let $M,N \in \mathbb{R}^{(k+n) \times (k+n)}$ be symmetric matrices and assume that there exists some matrix $\bar{Z} \in \mathbb{R}^{n \times k}$ such that
\begin{equation}
    \label{matSlater}
\begin{bmatrix} I \\ \bar{Z} \end{bmatrix}^\top N \begin{bmatrix} I \\ \bar{Z} \end{bmatrix} > 0.
\end{equation}
Then the following statements are equivalent:
\begin{enumerate}[(I)]
    \item 
    $
    \begin{bmatrix} I \\ Z \end{bmatrix}^\top M \begin{bmatrix} I \\ Z \end{bmatrix} \geq 0 \:\: \forall  Z \in \mathbb{R}^{n \times k} \text{ with} \begin{bmatrix} I \\ Z \end{bmatrix}^\top N \begin{bmatrix} I \\ Z \end{bmatrix} \geq 0.
    $ \label{mati}
    \item 
    $
    \begin{bmatrix} I \\ Z \end{bmatrix}^\top M \begin{bmatrix} I \\ Z \end{bmatrix} \geq 0 \:\: \forall  Z \in \mathbb{R}^{n \times k} \text{ with} \begin{bmatrix} I \\ Z \end{bmatrix}^\top N \begin{bmatrix} I \\ Z \end{bmatrix} > 0.
    $ \label{matii}
    \item There exists a scalar $\alpha \geq 0$ such that $M - \alpha N \geq 0$.\label{matiii}
\end{enumerate}
\end{theorem}

Note that for $k = 1$, the assumption \eqref{matSlater} reduces to the standard Slater condition. In this case, Theorem~\ref{matSlemma} recovers Theorems \ref{Slemma} and \ref{sSlemma} in the following sense: the equivalence of \eqref{mati} and \eqref{matiii} is the statement of Theorem~\ref{Slemma}. The equivalence of \eqref{matii} and \eqref{matiii} generalizes Theorem~\ref{sSlemma} for quadratic forms to general quadratic functions.

\begin{proof}[Proof of Theorem~\ref{matSlemma}] 
Clearly, \eqref{mati} $\implies$ \eqref{matii} and \eqref{matiii} $\implies$ \eqref{mati}. Thus, it suffices to prove that \eqref{matii} $\implies$ \eqref{matiii}. Our strategy will be to show that \eqref{matii} implies statement \eqref{homii} of Theorem~\ref{hommatSlemma}. To this end, suppose that \eqref{matii} holds and let $X \in \mathbb{R}^{(k+n) \times k}$ be such that $X^\top N X > 0$. Partition $X$ as 
$$
X = \begin{bmatrix}
X_1 \\ X_2
\end{bmatrix},
$$
where $X_1 \in \mathbb{R}^{k \times k}$ and $X_2 \in \mathbb{R}^{n \times k}$. Clearly, for all sufficiently small $\epsilon > 0$ we have
$$
\begin{bmatrix}
X_1 + \epsilon I \\ X_2
\end{bmatrix}^\top N \begin{bmatrix}
X_1 + \epsilon I \\ X_2
\end{bmatrix} > 0.
$$
Also note that $X_1 + \epsilon I$ is nonsingular for all sufficiently small $\epsilon > 0$. This implies that 
$$
\begin{bmatrix}
I \\ X_2 (X_1 + \epsilon I)^{-1}
\end{bmatrix}^\top N \begin{bmatrix}
I \\ X_2 (X_1 + \epsilon I)^{-1}
\end{bmatrix} > 0.
$$
By \eqref{matii}, we have
$$
\begin{bmatrix}
I \\ X_2 (X_1 + \epsilon I)^{-1}
\end{bmatrix}^\top M \begin{bmatrix}
I \\ X_2 (X_1 + \epsilon I)^{-1}
\end{bmatrix} \geq 0,
$$
equivalently,
$$
\begin{bmatrix}
X_1 + \epsilon I \\ X_2
\end{bmatrix}^\top M \begin{bmatrix}
X_1 + \epsilon I \\ X_2
\end{bmatrix} \geq 0
$$
for all $\epsilon > 0$ sufficiently small. By taking the limit $\epsilon \downarrow 0$ we conclude that $X^\top M X \geq 0$. Therefore, statement \eqref{homii} (equivalently, statement \eqref{homiii}) of Theorem~\ref{hommatSlemma} is satisfied. This means that \eqref{matiii} holds, which proves the theorem.
\end{proof}

As a special case of Theorem~\ref{matSlemma} we recover Theorem 3.3 of \cite{Luo2004}.
\begin{corollary}
The quadratic matrix inequality $M_{11} + M_{12} Z + Z^\top M_{12}^\top + Z^\top M_{22} Z \geq 0$ holds for all $Z \in \mathbb{R}^{n \times k}$ satisfying $I - Z^\top D Z \geq 0$ if and only if there exists a scalar $\alpha \geq 0$ such that 
$$
\begin{bmatrix}
M_{11} & M_{12} \\ M_{12}^\top & M_{22}
\end{bmatrix} - \alpha \begin{bmatrix}
I & 0 \\ 0 & -D
\end{bmatrix} \geq 0.
$$
\end{corollary}

\begin{proof}
Note that the generalized Slater condition \eqref{matSlater} is satisfied (one can choose e.g., $\bar{Z} = 0$). Thus, the statement follows from Theorem~\ref{matSlemma}.
\end{proof}

Theorem~\ref{matSlemma} provides a natural generalization of the S-lemma to matrix variables. However, note that for the application that we have in mind, we need a slightly different version of the theorem. Indeed, note that in the data-driven context of Section~\ref{sectionproblem}, a \emph{strict} inequality \eqref{ineqABPK} must hold for all $(A,B)$ satisfying a non-strict inequality \eqref{ineqAB}. As such, we need to extend Theorem~\ref{matSlemma} to the case when the inequality involving $M$ is \emph{strict}. Before we do so we introduce the shorthand notation
$$
\mathcal{S}_N := \left\{ Z \in \mathbb{R}^{n \times k} \mid \begin{bmatrix} I \\ Z \end{bmatrix}^\top N \begin{bmatrix} I \\ Z \end{bmatrix} \geq 0 \right\}.
$$

\begin{theorem}[Strict matrix S-lemma]
\label{strictmatSlemma}
Let $M$ and $N$ by symmetric matrices in  $\mathbb{R}^{(k+n) \times (k+n)}$. Assume that $\mathcal{S}_N$ is bounded and that there exists some matrix $\bar{Z} \in \mathbb{R}^{n \times k}$ satisfying \eqref{matSlater}. Then we have that
    \begin{equation}
        \label{strictM}
    \begin{bmatrix} I \\ Z \end{bmatrix}^\top M \begin{bmatrix} I \\ Z \end{bmatrix} > 0 \: \text{ for all }  Z \in \mathcal{S}_N
    \end{equation}
    if and only if there exists $\alpha \geq 0$ such that $M- \alpha N > 0$.
\end{theorem}

\begin{proof}
The ``if" part is clear, so we focus on proving the ``only if" part. Suppose that \eqref{strictM} holds. We claim that there exists an $\epsilon > 0$ such that 
\begin{equation}
    \label{Mepsilon}
\begin{bmatrix} I \\ Z \end{bmatrix}^\top (M-\epsilon I) \begin{bmatrix} I \\ Z \end{bmatrix} > 0 \: \: \text{for all } Z \in \mathcal{S}_N.
\end{equation}
Suppose that this is not the case. Then there exists a sequence $\{\epsilon_i\}$ with $\epsilon_i \to 0$ as $i \to \infty$ with the property that for each $i$ there exists $Z_i \in \mathcal{S}_N$ such that
$$
\begin{bmatrix} I \\ Z_i \end{bmatrix}^\top (M-\epsilon_i I) \begin{bmatrix} I \\ Z_i \end{bmatrix} \not > 0.
$$
Since $\mathcal{S}_N$ is bounded, the sequence $\{Z_i\}$ is bounded. As such, by the Bolzano-Weierstrass theorem, it contains a converging subsequence with limit, say, $Z^*$. We conclude that
$$
\begin{bmatrix} I \\ Z^* \end{bmatrix}^\top M \begin{bmatrix} I \\ Z^* \end{bmatrix} \not > 0.
$$
Note that $\mathcal{S}_N$ is closed and thus $Z^* \in \mathcal{S}_N$. Since \eqref{strictM} holds we arrive at a contradiction. Therefore, we conclude that there exists an $\epsilon > 0$ such that \eqref{Mepsilon} holds. In particular, this implies the existence of $\epsilon > 0$ such that  
$$
\begin{bmatrix} I \\ Z \end{bmatrix}^\top (M-\epsilon I) \begin{bmatrix} I \\ Z \end{bmatrix} \geq 0 \: \: \text{for all } Z \in \mathcal{S}_N.
$$
Now, by Theorem~\ref{matSlemma} there exists an $\alpha \geq 0$ such that 
$$
(M-\epsilon I) -\alpha N \geq 0.
$$
We conclude that $M-\alpha N > 0$ which proves the theorem.
\end{proof}

At this point, it is worthwhile to point out a relation between Theorem~\ref{strictmatSlemma} and the literature on LMI relaxations in robust control, see \cite{Scherer2005,Scherer2006,Scherer2006b}. In fact, we can derive a type of matrix S-lemma from the general theory in \cite{Scherer2005}. As it turns out, this matrix S-lemma is also a corollary of Theorem \ref{strictmatSlemma}. To obtain the result, we substitute $A = 0$, $B = I$, $W(x) = -M$ and 
$${\bf\Delta} = \left\{\Delta \mid \begin{bmatrix}
	\Delta \\ I 
	\end{bmatrix}^\top N \begin{bmatrix}
	\Delta \\ I 
	\end{bmatrix}\geq 0\right\},
$$ 
into Equation 1.2 of \cite{Scherer2005}. Then, we combine the full block S-procedure (c.f. \cite{Scherer2001},\cite[page 367]{Scherer2005}) with the fact that the LMI relaxation of \cite{Scherer2005} is exact for a single full block \cite[Thm. 5.3]{Scherer2005}. This yields the following result. 
\begin{corollary}
\label{corollaryScherer}
Let $M,N \in \mathbb{R}^{(k+n) \times (k+n)}$ be symmetric matrices, partitioned as in \eqref{partitionMN}. Assume that $N$ is nonsingular, $N_{11} \geq 0$ and $N_{22} < 0$. Then we have that
\begin{equation*}
\begin{bmatrix} I \\ Z \end{bmatrix}^\top M \begin{bmatrix} I \\ Z \end{bmatrix} > 0 \: \text{ for all }  Z \in \mathcal{S}_N
\end{equation*}
if and only if there exists $\alpha \geq 0$ such that $M - \alpha N > 0$.
\end{corollary}
\begin{proof}
We will show that the assumptions on $N$ imply the assumptions of Theorem \ref{strictmatSlemma}. First of all, $N_{22} < 0$ implies that $\mathcal{S}_N$ is bounded. Secondly, the nonsingularity of $N$ and $N_{22}$ imply that the Schur complement $N_{11} - N_{12} N_{22}^{-1} N_{12}^\top$ is nonsingular, and since $N_{11} \geq 0$ and $N_{22} < 0$ we have $N_{11} - N_{12} N_{22}^{-1} N_{12}^\top > 0$. Thus, the matrix $\bar{Z} := - N_{22}^{-1} N_{12}^\top$ satisfies the Slater condition \eqref{matSlater}. The result now follows from Theorem \ref{strictmatSlemma}.
\end{proof}

It turns out that we can even state Theorem~\ref{strictmatSlemma} without the boundedness assumption if some more structure on the matrices $M$ and $N$ is given. In fact, we have the following result. 
\begin{theorem}
\label{strictmatSlemma2}
Let $M,N \in \mathbb{R}^{(k+n) \times (k+n)}$ be symmetric matrices, partitioned as in \eqref{partitionMN}. Assume that $M_{22} \leq 0$, $N_{22} \leq 0$ and $\ker N_{22} \subseteq \ker N_{12}$. Suppose that there exists some matrix $\bar{Z} \in \mathbb{R}^{n \times k}$ satisfying \eqref{matSlater}. Then we have that
    \begin{equation}
        \label{strictM2}
    \begin{bmatrix} I \\ Z \end{bmatrix}^\top M \begin{bmatrix} I \\ Z \end{bmatrix} > 0 \: \text{ for all }  Z \in \mathcal{S}_N
    \end{equation}
    if and only if there exist $\alpha \geq 0$ and $\beta > 0$ such that 
    $$
    M- \alpha N \geq \begin{bmatrix}
    \beta I & 0 \\ 0 & 0
    \end{bmatrix}.
    $$
\end{theorem}

\begin{proof}
The ``if" part is clear so we focus on proving the ``only if" statement. Suppose that \eqref{strictM2} holds. We will first prove that $\ker N_{22} \subseteq \ker M_{22}$ and $\ker N_{22} \subseteq \ker M_{12}$. Let $Z \in \mathcal{S}_N$ and $\hat{Z} \in \mathbb{R}^{n \times k}$ be such that $N_{22} \hat{Z} = 0$. By the hypothesis $\ker N_{22} \subseteq \ker N_{12}$ we have $Z + \gamma \hat{Z} \in \mathcal{S}_N$ for any $\gamma \in \mathbb{R}$. Thus, we obtain
\begin{equation}
    \label{ZZhat}
\begin{bmatrix} I \\ Z \end{bmatrix}^\top M \begin{bmatrix} I \\ Z \end{bmatrix} + \gamma (M_{12}\hat{Z} + (M_{12}\hat{Z})^\top) + \gamma^2 \hat{Z}^\top M_{22} \hat{Z} > 0.
\end{equation}
This implies that $M_{22} \hat{Z} = 0$. Indeed, recall that $M_{22} \leq 0$. Thus, if $M_{22} \hat{Z} \neq 0$ then there exists a sufficiently large $\gamma$ such that \eqref{ZZhat} is violated. Similarly, we conclude that $M_{12} \hat{Z} = 0$. Therefore, we have shown that
\begin{equation}
    \label{inclusionskerMN}
    \ker N_{22} \subseteq \ker M_{22} \text{ and } \ker N_{22} \subseteq \ker M_{12}.
\end{equation}
Subsequently, we claim that there exists a $\beta > 0$ such that 
\begin{equation}
    \label{Mbeta}
\begin{bmatrix} I \\ Z \end{bmatrix}^\top \left(M- \begin{bmatrix}
\beta I & 0 \\ 0 & 0
\end{bmatrix} \right) \begin{bmatrix} I \\ Z \end{bmatrix} > 0 \: \: \text{for all } Z \in \mathcal{S}_N.
\end{equation}
If this claim is not true, then there exists a sequence $\{\beta_i\}$ such that $\beta_i \to 0$ and for all $i$ there exists $Z_i \in \mathcal{S}_N$ such that 
\begin{equation}
    \label{MbetanotPD}
\begin{bmatrix} I \\ Z_i \end{bmatrix}^\top \left( M-\begin{bmatrix}
\beta_i I & 0 \\ 0 & 0
\end{bmatrix} \right) \begin{bmatrix} I \\ Z_i \end{bmatrix} \not > 0.
\end{equation}

\begin{figure*}[b!]
\normalsize
\vspace*{4pt}
\hrulefill
\begin{equation}
\begin{bmatrix}
    P-\beta I & 0 & 0 & 0 \\
    0 & -P & -L^\top & 0 \\
    0 & -L & 0 & L \\
    0 & 0 & L^\top & P
    \end{bmatrix} - \alpha \begin{bmatrix}
    I & X_+ \\ 0 & -X_- \\ 0 & -U_- \\ 0 & 0
    \end{bmatrix}
    \begin{bmatrix}
    \Phi_{11} & \Phi_{12} \\
    \Phi_{12}^\top & \Phi_{22}
    \end{bmatrix}
    \begin{bmatrix}
    I & X_+ \\ 0 & -X_- \\ 0 & -U_- \\ 0 & 0
    \end{bmatrix}^\top \geq 0. \label{LMIstab}\tag{FS}
\end{equation}
\end{figure*}

Define $\mathcal{V} := \{ Z \in \mathbb{R}^{n \times k} \mid N_{22} Z = 0 \}$. Write $Z_i$ as $Z_i = Z_i^1 + Z_i^2$, where $Z_i^1 \in \mathcal{V}^\perp$ and $Z_i^2 \in \mathcal{V}$. By the hypothesis $\ker N_{22} \subseteq \ker N_{12}$ we see that $Z_i^1 \in \mathcal{S}_N$. Next, we claim that the sequence $\{ Z_i^1 \}$ is bounded. We will prove this claim by contradiction. Thus, suppose that $\{ Z_i^1 \}$ is unbounded. Clearly, the sequence 
$$
\left\{ \frac{Z_i^1}{\norm{Z_i^1}} \right\}
$$
is bounded. By the Bolzano-Weierstrass theorem it thus has a convergent subsequence with limit, say $Z_*$. Note that
$$
\frac{1}{\norm{Z_i^1}^2}\left( N_{11} + N_{12} Z_i^1 + (N_{12} Z_i^1)^\top + (Z_i^1)^\top N_{22} Z_i^1 \right) \geq 0.
$$
By taking the limit along the subsequence as $i \to \infty$, we get $Z_*^\top N_{22} Z_* \geq 0$. Using the fact that $N_{22} \leq 0$ we conclude that $Z_* \in \mathcal{V}$. Since $Z_i^1 \in \mathcal{V}^\perp$ for all $i$, also $\frac{Z_i^1}{\norm{Z_i^1}} \in \mathcal{V}^\perp$ and thus $Z_* \in \mathcal{V}^\perp$. Therefore, we conclude that both $Z_* \in \mathcal{V}$ and $Z_* \in \mathcal{V}^\perp$, i.e., $Z_* = 0$. This is a contradiction since $\frac{Z_i^1}{\norm{Z_i^1}}$ has norm 1 for all $i$. We conclude that the sequence $\{Z_i^1\}$ is bounded. It thus contains a convergent subsequence with limit, say $Z_{\infty}$. Note that $\mathcal{S}_N$ is closed and thus $Z_{\infty} \in \mathcal{S}_N$. By \eqref{inclusionskerMN} and \eqref{MbetanotPD} we conclude that
$$
\begin{bmatrix} I \\ Z_i^1 \end{bmatrix}^\top \left( M-\begin{bmatrix}
\beta_i I & 0 \\ 0 & 0
\end{bmatrix} \right) \begin{bmatrix} I \\ Z_i^1 \end{bmatrix} \not > 0
$$
for all $i$. We take the limit as $i \to \infty$, which yields  
$$
\begin{bmatrix} I \\ Z_\infty \end{bmatrix}^\top M \begin{bmatrix} I \\ Z_\infty \end{bmatrix} \not > 0.
$$
As $Z_{\infty} \in \mathcal{S}_N$ this contradicts \eqref{strictM2}. As such, we conclude that there exists $\beta > 0$ such that \eqref{Mbeta} holds. In particular, there exists $\beta > 0$ such that 
$$
\begin{bmatrix} I \\ Z \end{bmatrix}^\top \left(M- \begin{bmatrix}
\beta I & 0 \\ 0 & 0
\end{bmatrix} \right) \begin{bmatrix} I \\ Z \end{bmatrix} \geq 0 \: \: \text{for all } Z \in \mathcal{S}_N.
$$
The theorem now follows by application of Theorem~\ref{matSlemma}.
\end{proof}

\section{Data-driven stabilization revisited}
\label{sectionddstab}
In this section, we apply the theory from Section~\ref{sectionSlemma} to data-driven stabilization, i.e., to Problems \ref{problem1} and \ref{problem2} defined in Section~\ref{sectionproblem}. To this end, for given $P = P^\top > 0$ and $K$ we define the partitioned matrices 
\begin{align}
    M =& 
\begin{pmat}[{|}]
M_{11} & M_{12} \cr\-
M_{12}^\top & M_{22} \cr
\end{pmat} := \begin{pmat}[{|.}]
    P & 0 & 0 \cr\-
    0 & -P & -PK^\top \cr
    0 & -KP & -KPK^\top \cr
    \end{pmat}, \label{Mstab} 
    \\
    \label{Nstab}
N =& \begin{pmat}[{|}]
N_{11} & N_{12} \cr\- N_{12}^\top & N_{22} \cr
\end{pmat} \nonumber \\
:=& \begin{pmat}[{.}]
    I & X_+ \cr\- 0 & -X_- \cr 0 & -U_- \cr
    \end{pmat}
    \begin{bmatrix}
    \Phi_{11} & \Phi_{12} \\
    \Phi_{12}^\top & \Phi_{22}
    \end{bmatrix}
    \begin{pmat}[{.}]
    I & X_+ \cr\- 0 & -X_- \cr 0 & -U_- \cr
    \end{pmat}^\top. 
\end{align}
Recall from Section \ref{sectionproblem} that data-driven stabilization entails deciding whether \eqref{ineqABPK} holds for all $(A,B)$ satisfying \eqref{ineqAB}. In terms of the matrices $M$ and $N$ as defined above, we thus have to decide whether 
\begin{equation}
\label{implicationMNdata}
\begin{bmatrix} I \\ Z \end{bmatrix}^\top M \begin{bmatrix} I \\ Z \end{bmatrix} > 0 \:\: \forall  Z \in \mathbb{R}^{(n+m) \times n} \text{ with} \begin{bmatrix} I \\ Z \end{bmatrix}^\top N \begin{bmatrix} I \\ Z \end{bmatrix} \geq 0.
\end{equation}
Here $Z$ is given by 
$$
Z := \begin{bmatrix}
A^\top \\ B^\top 
\end{bmatrix}.
$$
The idea is now to apply Theorem~\ref{strictmatSlemma2}. To this end, we have to verify its assumptions. In particular, we will check that $M_{22} \leq 0$, $N_{22} \leq 0$ and $\ker N_{22} \subseteq \ker N_{12}$. Note that 
$$
M_{22} = - \begin{bmatrix}
I \\ K
\end{bmatrix} P \begin{bmatrix}
I \\ K
\end{bmatrix}^\top \leq 0, \: N_{22} = \begin{bmatrix}
X_- \\ U_-
\end{bmatrix} \Phi_{22} \begin{bmatrix}
X_- \\ U_-
\end{bmatrix}^\top \leq 0,
$$
because $P > 0$ and $\Phi_{22} < 0$. Since $\Phi_{22}$ is nonsingular, we also see that
\begin{align*}
\ker N_{22} &= \ker \begin{bmatrix}
X_- \\ U_-
\end{bmatrix}^\top, \\
\ker N_{12} &= \ker \bigg( (\Phi_{12}+X_+ \Phi_{22})\begin{bmatrix}
X_- \\ U_-
\end{bmatrix}^\top \bigg),
\end{align*}
and thus $\ker N_{22} \subseteq \ker N_{12}$. We conclude that the assumptions of Theorem~\ref{strictmatSlemma2} are satisfied. We assume that the generalized Slater condition \eqref{matSlater} holds for $N$ in \eqref{Nstab}. Then, Theorem~\ref{strictmatSlemma2} asserts that \eqref{implicationMNdata} holds if and only if there exist scalars $\alpha \geq 0$ and $\beta > 0$ such that
\begin{equation}
\label{ineqMNab}
M - \alpha N \geq \begin{bmatrix}
\beta I & 0  \\ 0 & 0 
\end{bmatrix}.
\end{equation}
From a design point of view, the matrices $P$ and $K$ that appear in $M$ are not given. However, the idea is now to \emph{compute} matrices $P$, $K$ and scalars $\alpha$ and $\beta$ such that \eqref{ineqMNab} holds. In fact, by the above discussion, the data $(U_-,X)$ are informative for quadratic stabilization \emph{if and only if} there exists an $n \times n$ matrix $P = P^\top > 0$, a $K \in \mathbb{R}^{m \times n}$ and two scalars $\alpha \geq 0$ and $\beta > 0$ such that \eqref{ineqMNab} holds. We note that \eqref{ineqMNab} (in particular, $M$) is not linear in $P$ and $K$. Nonetheless, by a rather standard change of variables and a Schur complement argument, we can transform \eqref{ineqMNab} into a linear matrix inequality. We summarize our progress in the following theorem, which is the main result of this section. 
\begin{theorem}
\label{theoremstab}
Assume that the generalized Slater condition \eqref{matSlater} holds for $N$ in \eqref{Nstab} and some $\bar{Z} \in \mathbb{R}^{(n+m)\times n}$. Then the data $(U_-,X)$ are informative for quadratic stabilization if and only if there exists an $n\times n$ matrix $P = P^\top > 0$, an $L \in \mathbb{R}^{m \times n}$ and scalars $\alpha \geq 0$ and $\beta > 0$ satisfying \eqref{LMIstab}.

Moreover, if $P$ and $L$ satisfy \eqref{LMIstab} then $K := L P^{-1}$ is a stabilizing feedback gain for all $(A,B) \in \Sigma$.
\end{theorem}

\begin{proof}
To prove the ``if" statement, suppose that there exist $P$, $L$, $\alpha$ and $\beta$ satisfying \eqref{LMIstab}. Define $K := L P^{-1}$. By computing the Schur complement of \eqref{LMIstab} with respect to its fourth diagonal block, we obtain \eqref{ineqMNab}. As such, \eqref{implicationMNdata} holds. We conclude that the data $(U_-,X)$ are informative for quadratic stabilization and $K = LP^{-1}$ is indeed a stabilizing controller for all $(A,B) \in \Sigma$. 

Conversely, to prove the ``only if" statement, suppose that the data $(U_-,X)$ are informative for quadratic stabilization. This means that there exist $P = P^\top > 0$ and $K$ such that \eqref{implicationMNdata} holds. By Theorem \ref{strictmatSlemma2} there exist $\alpha \geq 0$ and $\beta >0$ satisfying \eqref{ineqMNab}. Finally, by defining $L := KP$ and using a Schur complement argument, we conclude that \eqref{LMIstab} is feasible. 
\end{proof}

Theorem~\ref{theoremstab} provides a powerful necessary \emph{and} sufficient condition under which quadratically stabilizing controllers can be obtained from noisy data. The assumption \eqref{matSlater} puts a mild condition on the data matrices appearing in \eqref{Nstab}. It is satisfied whenever $N$ has at least $n$ positive eigenvalues, a condition that is simple to verify from given data. So far, this condition was satisfied in all of our numerical experiments, see Section~\ref{sectionsimulations} for more details\footnote{In addition, we remark that even if the generalized Slater condition does not hold, the `if' statement of Theorem \ref{theoremstab} remains true.}. Theorem~\ref{theoremstab} leads to an effective design procedure for obtaining stabilizing controllers directly from data. Indeed, the approach entails solving the linear matrix inequality \eqref{LMIstab} for $P, L, \alpha$ and $\beta$ and computing a controller as $K = LP^{-1}$. We now discuss some of the features of our control design procedure.
\begin{enumerate}
    \item First of all, we note that the procedure is \emph{non-conservative} since Theorem~\ref{theoremstab} provides a necessary and sufficient condition for obtaining quadratically stabilizing controllers from data. 
    \item We believe that our approach based on the set $\Sigma$ of \emph{open-loop} systems provides a valuable alternative to the data-based closed-loop system parameterizations of \cite[Thm. 2]{DePersis2020} and \cite[Thm. 4]{Berberich2019c}. Indeed, in the case of noisy data, it was recognized that certain linear constraints \cite[Eq. (3)]{Berberich2019c} defining these closed-loop systems were difficult to incorporate in the control design\footnote{In fact, it was mentioned in \cite{Koch2020} that involving the condition \cite[Eq. (3)]{Berberich2019c} in design procedures is still an open problem.}. Our design procedure does not suffer from the above problem. In fact, the constraint \cite[Eq. (3)]{Berberich2019c} is \emph{automatically incorporated} in our control design approach.
    \item The variables $P, L, \alpha$ and $\beta$ are \emph{independent} of the time horizon $T$ of the experiment. In fact, note that $P \in \mathbb{R}^{n \times n}$, $L \in \mathbb{R}^{m \times n}$ and $\alpha,\beta \in \mathbb{R}$. Also, the LMI \eqref{LMIstab} is of dimension $(3n+m) \times (3n+m)$ and thus independent of $T$. As such, our approach fundamentally differs from the design methods in \cite{DePersis2020,Berberich2019c} where certain decision variables have dimension $T \times n$, c.f. \cite[Thm. 6]{DePersis2020} and \cite[Cor. 6]{Berberich2019c}. We believe that our $T$-independent design method will play a crucial role in control design from larger data sets. We note that the collection of big data sets is often unavoidable, for example because the signal-to-noise ratio is small, or because the data-generating system is large-scale. 
\end{enumerate}

\begin{remark}
We note that under the extra assumption 
\begin{equation}
    \label{fullrank}
\rank \begin{bmatrix}
X_- \\ U_-
\end{bmatrix} = n+m
\end{equation}
it is possible to prove a variant Theorem~\ref{theoremstab} in which the non-strict inequality is replaced by a strict inequality, and the term $-\beta I$ is removed. This can be done by invoking Theorem~\ref{strictmatSlemma}, which is possible since \eqref{fullrank} implies that the set $\Sigma$ is bounded. The reason is that the coefficient matrix $N_{22}$ defining the quadratic term in \eqref{ineqAB} is negative definite if \eqref{fullrank} holds. 

Here we chose to state and prove Theorem~\ref{theoremstab} in the slightly more general setting without assuming \eqref{fullrank}. In the discussion preceding Theorem~\ref{theoremstab} we have verified the assumptions of Theorem~\ref{strictmatSlemma2} for $M$ and $N$ in \eqref{Mstab}, \eqref{Nstab}. In particular, this implies that the subspace inclusions \eqref{inclusionskerMN} hold and thus $\ker \begin{bmatrix}
X_-^\top & U_-^\top 
\end{bmatrix} \subseteq \ker \begin{bmatrix}
I & K^\top
\end{bmatrix}$, equivalently
\begin{equation}
    \label{IKXU}
\im \begin{bmatrix}
I \\ K
\end{bmatrix} \subseteq 
\im
\begin{bmatrix}
X_- \\ U_-
\end{bmatrix}.
\end{equation}
Therefore, any controller $K$ that stabilizes the systems in $\Sigma$ is necessarily of the form \eqref{IKXU}. This generalizes \cite[Lem. 15]{vanWaarde2020} to the case of noisy data.
\end{remark}

\section{Inclusion of performance specifications}
\label{sectionperformance}
In this section we extend our data-driven stabilization result by including different performance specifications. In particular, we will treat the $\mathcal{H}_2$ and $\mathcal{H}_\infty$ control problems, thereby illustrating the general applicability of the theory in Section~\ref{sectionSlemma}. 

\subsection{$\mathcal{H}_2$ control}
As before, consider the the unknown system \eqref{system1}. We associate to \eqref{system1} a performance output
\begin{equation}
    \bmz(t) = C\bmx(t) + D\bmu(t), \label{CD}
\end{equation}
where $\bmz \in \mathbb{R}^p$, and $C$ and $D$ are known matrices that specify the performance. For any $(A,B) \in \Sigma$ explaining the data, the feedback law $\bmu = K \bmx$ yields the closed-loop system 
\begin{equation}
\label{closedloop}
\begin{aligned}
    \bmx(t+1) &= (A+BK)\bmx(t) + \bmw(t) \\
    \bmz(t) &= (C+DK)\bmx(t).
\end{aligned}
\end{equation}
The transfer matrix from $\bmw$ to $\bmz$ of \eqref{closedloop} is given by
$$
G(z) := (C+DK) (zI-(A+BK))^{-1},
$$
and its $\mathcal{H}_2$ norm is denoted by $\norm{G(z)}_{\mathcal{H}_2}$. Let $\gamma > 0$. It is well-known that $A+BK$ is stable and $\norm{G(z)}_{\mathcal{H}_2} < \gamma$ if and only if there exists a matrix $P = P^\top > 0$ such that 

\begin{equation}
\begin{aligned}
\label{H2iff}
   P  &> (A+BK)^\top P (A+BK) + (C+DK)^\top (C+DK) \\
   \trace P &< \gamma^2,
\end{aligned}
\end{equation}
where $\trace$ denotes trace. The data-driven $\mathcal{H}_2$ problem entails the computation of a feedback gain $K$ from data such that $\norm{G(z)}_{\mathcal{H}_2} < \gamma$ \emph{for all} $(A,B) \in \Sigma$. Similar to our results for quadratic stabilization, we restrict the attention to a matrix $P$ that is common for all $(A,B)$. This leads to the following natural definition. 

\begin{definition}
Assume that $(U_-,X)$ satisfies \eqref{eqdatanoise} for some $W_-$ satisfying Assumption \ref{assumption1}.
The data $(U_-,X)$ are \emph{informative for $\mathcal{H}_2$ control} with performance $\gamma$ if there exist matrices $P = P^\top > 0$ and $K$ such that \eqref{H2iff} holds for all $(A,B) \in \Sigma$.
\end{definition}

With the theory of Section~\ref{sectionSlemma} in place, characterizing informativity for $\mathcal{H}_2$ control essentially boils down to massaging the inequalities \eqref{H2iff} such that they are amenable to design. To this end, note that the first inequality of \eqref{H2iff} is equivalent to 
$$
Y - A_{Y,L}^\top P A_{Y,L} - C_{Y,L}^\top C_{Y,L} > 0,
$$
where we defined $A_{Y,L} := AY + BL$ and $C_{Y,L} := CY + DL$ with $Y := P^{-1}$ and $L := KY$. Using a Schur complement argument, this is equivalent to
\begin{equation}
    \label{eqYLP}
\begin{bmatrix}
Y - C_{Y,L}^\top C_{Y,L} & A_{Y,L}^\top \\ 
A_{Y,L} & Y
\end{bmatrix} > 0,
\end{equation}
Now, \eqref{eqYLP} holds if and only if 
\begin{align}
    Y - C_{Y,L}^\top C_{Y,L} &> 0, \label{CLPCL}\\
    Y - A_{Y,L} (Y-C_{Y,L}^\top C_{Y,L})^{-1} A_{Y,L}^\top &> 0. \label{ALPAL}
\end{align}
Note that \eqref{CLPCL} is independent of $A$ and $B$. In turn, we can write \eqref{ALPAL} as
\begin{equation}
    \label{ineqABYL}
\begin{pmat}[{.}]
I \cr\- A^\top \cr B^\top \cr
\end{pmat}^\top 
\hspace{-4pt}
\underbrace{
\begin{pmat}[{|}]
Y & 0 \cr\-
0 & - \begin{bmatrix}
Y \\ L
\end{bmatrix}
(Y - C_{Y,L}^\top C_{Y,L})^{-1}
\begin{bmatrix}
Y \\ L
\end{bmatrix}^\top \cr
\end{pmat}
}_{=:M}
\hspace{-4pt}
\begin{pmat}[{}]
I \cr\- A^\top \cr B^\top \cr
\end{pmat} \hspace{-1pt} > \hspace{-1pt} 0.
\end{equation}
Note that the inequality \eqref{ineqABYL} is of a form where $A$ and $B$ appear on the left and their transposes appear on the right, analogous to \eqref{ineqABPK}. As such, we are in a position to apply Theorem~\ref{strictmatSlemma2}. In fact, we derive the following theorem.

\begin{theorem}
\label{theoremH2}
Assume that the generalized Slater condition \eqref{matSlater} holds for $N$ in \eqref{Nstab} and some $\bar{Z} \in \mathbb{R}^{(n+m)\times n}$. Then the data $(U_-,X)$ are informative for $\mathcal{H}_2$ control with performance $\gamma$ if and only if there exist matrices $Y = Y^\top > 0$, $Z = Z^\top$ and $L$, and scalars $\alpha \geq 0$ and $\beta > 0$ satisfying \eqref{LMIH2}.

Moreover, if $Y$ and $L$ satisfy \eqref{LMIH2} then $K := L Y^{-1}$ is such that $A+BK$ is stable and $\norm{G(z)}_{\mathcal{H}_2} < \gamma$ for all $(A,B) \in \Sigma$.
\end{theorem}

\begin{figure*}[b!]
\normalsize
\vspace*{4pt}
\hrulefill
\begin{equation}
\begin{bmatrix}
    Y- \beta I & 0 & 0 & 0 & 0 \\
    0 & 0 & 0 & Y & 0  \\
    0 & 0 & 0 & L & 0 \\
    0 & Y & L^\top & Y & C_{Y,L}^\top \\
    0 & 0 & 0 & C_{Y,L} & I
    \end{bmatrix} - \alpha \begin{bmatrix}
    I & X_+ \\ 0 & -X_- \\ 0 & -U_- \\ 0 & 0 \\ 0 & 0
    \end{bmatrix}
    \begin{bmatrix}
    \Phi_{11} & \Phi_{12} \\
    \Phi_{12}^\top & \Phi_{22}
    \end{bmatrix}
    \begin{bmatrix}
    I & X_+ \\ 0 & -X_- \\ 0 & -U_- \\ 0 & 0 \\ 0 & 0
    \end{bmatrix}^\top \geq 0, \quad 
    \begin{bmatrix}
    Y & C_{Y,L}^\top \\
    C_{Y,L} & I
    \end{bmatrix} > 0, \quad 
    \begin{matrix*}[r]
    \begin{bmatrix}
    Z & I \\ I & Y
    \end{bmatrix} \geq 0, \\[3mm]
    \trace Z < \gamma^2.
    \end{matrix*}
    \label{LMIH2}\tag{$\mathcal{H}_2$}
\end{equation}
\end{figure*}

\begin{proof}
Suppose that \eqref{LMIH2} is feasible and define $P:= Y^{-1}$ and $K := LP$. The last two inequalities of \eqref{LMIH2} imply that $\trace P < \gamma^2$. We now compute the Schur complement of the first LMI in \eqref{LMIH2} with respect to the diagonal block 
$$
\begin{bmatrix}
    Y & C_{Y,L}^\top \\
    C_{Y,L} & I
    \end{bmatrix}.
$$
We thereby make use of the fact that this block is nonsingular by the second LMI of \eqref{LMIH2}. The computation of the Schur complement results in
\begin{equation}
\label{alphabetaH2}
M - \alpha N \geq \begin{bmatrix}
\beta I & 0 \\ 0 & 0
\end{bmatrix},
\end{equation}
where $M$ is defined in \eqref{ineqABYL} and $N$ is defined in \eqref{Nstab}. We thus conclude that the inequality \eqref{ineqABYL} is satisfied for all $(A,B) \in \Sigma$. As such, \eqref{ALPAL} holds for all $(A,B) \in \Sigma$. Note that \eqref{CLPCL} holds by the second LMI of \eqref{LMIH2}. Therefore, we conclude that \eqref{H2iff} holds for all $(A,B) \in \Sigma$. In other words, the data $(U_-,X)$ are informative for $\mathcal{H}_2$ control with performance $\gamma$, and $K = LY^{-1}$ is a suitable controller. 

Conversely, suppose that the data $(U_-,X)$ are informative for $\mathcal{H}_2$ control with performance $\gamma$. Then there exist matrices $P = P^\top > 0$ and $K$ such that \eqref{H2iff} holds for all $(A,B) \in \Sigma$. Define $Y:=P^{-1}$, $L:=KY$ and $Z:=P$. Clearly, the last two inequalities of \eqref{LMIH2} are satisfied by definition of $Z$. In addition, we know that \eqref{CLPCL} and \eqref{ALPAL} hold for all $(A,B) \in \Sigma$.  
By \eqref{CLPCL}, the second LMI of \eqref{LMIH2} is satisfied. To prove that the first LMI of \eqref{LMIH2} also holds, we want to apply Theorem~\ref{strictmatSlemma2}. Note that we have already verified the assumptions of this theorem for the matrix $N$ in \eqref{Nstab}, see the discussion preceding Theorem~\ref{theoremstab}. In addition, we note that
$$
M_{22} = - \begin{bmatrix}
Y \\ L
\end{bmatrix}
(Y - C_{Y,L}^\top C_{Y,L})^{-1}
\begin{bmatrix}
Y \\ L
\end{bmatrix}^\top \leq 0
$$
since $Y - C_{Y,L}^\top C_{Y,L} > 0$. Hence, Theorem~\ref{strictmatSlemma2} is applicable. We conclude that there exist $\alpha \geq 0$ and $\beta > 0$ such that \eqref{alphabetaH2} holds. Using a Schur complement argument, we see that $Y$, $L$, $\alpha$ and $\beta$ satisfy the first LMI of \eqref{LMIH2}. Thus, \eqref{LMIH2} is feasible which proves the theorem. 
\end{proof}

\begin{remark}
If we know a priori that the noise $\bmw$ is contained in a subspace, say $\im E$, then this information can easily be exploited in the $\mathcal{H}_2$ controller design. In fact, we only need to replace the LMI involving $Z$ by
$$
\begin{bmatrix}
Z & E^\top \\
E & Y
\end{bmatrix} \geq 0. 
$$
We recall that prior knowledge of $\bmw \in \im E$, if available, can also be captured by our noise model, see Remark \ref{remarkE}. A natural choice is thus to use $E$ both in the noise model \eqref{asnoise} as well as in the LMI \eqref{LMIH2}. However, we remark that this is not necessary: the noise in the experiment may come from a different subspace than the disturbances that are attenuated by the $\mathcal{H}_2$ controller. 
\end{remark}

\subsection{$\mathcal{H}_\infty$ control}
In this section we will turn our attention to the $\mathcal{H}_{\infty}$ control problem. As before, consider system \eqref{system1} with performance output \eqref{CD}. For any $(A,B) \in \Sigma$, the feedback $\bmu = K\bmx$ yields the system \eqref{closedloop} with transfer matrix from $\bmw$ to $\bmz$ given by $G(z)$. We will denote the $\mathcal{H}_{\infty}$ norm of $G(z)$ by $\norm{G(z)}_{\mathcal{H}_\infty}$. Let $\gamma > 0$. By \cite[Thm. 4.6.6(iii)]{Skelton1997}, the matrix $A+BK$ is stable and $\norm{G(z)}_{\mathcal{H}_\infty} < \gamma$ if and only if there exists a matrix $P = P^\top > 0$ such that 
\begin{align}
    P - A_K^\top (P^{-1} -\frac{1}{\gamma^2} I)^{-1} A_K - C_K^\top C_K &> 0, \label{Hinf1} \\
    P^{-1} -\frac{1}{\gamma^2} I &> 0, \label{Hinf2}
\end{align}
where we have defined $A_K := A+BK$ and $C_K := C+DK$. We now have the following definition of informativity for $\mathcal{H}_{\infty}$ control. 
\begin{definition}
Assume that $(U_-,X)$ satisfies \eqref{eqdatanoise} for some $W_-$ satisfying Assumption \ref{assumption1}.
The data $(U_-,X)$ are \emph{informative for $\mathcal{H}_\infty$ control} with performance $\gamma$ if there exist matrices $P = P^\top > 0$ and $K$ such that \eqref{Hinf1} and \eqref{Hinf2} hold for all $(A,B) \in \Sigma$.
\end{definition}

By pre- and postmultiplication of \eqref{Hinf1} by $P^{-1}$ we obtain 
\begin{align*}
    Y - A_{Y,L}^\top (Y -\frac{1}{\gamma^2} I)^{-1} A_{Y,L} - C_{Y,L}^\top C_{Y,L} &> 0, \\
    Y -\frac{1}{\gamma^2} I &> 0,
\end{align*}
where the matrices $Y := P^{-1}$, $L := KY$, $A_{Y,L} := AY + BL$ and $C_{Y,L}:=CY+DL$ are defined as in the $\mathcal{H}_2$ problem. Note that the first of these inequalities can again be written in the -by now familiar- form
\begin{equation*}
\begin{pmat}[{}]
I \cr\- A^\top \cr B^\top \cr
\end{pmat}^\top 
\begin{pmat}[{|}]
Y - C_{Y,L}^\top C_{Y,L} & 0 \cr\-
0 & - \begin{bmatrix}
Y \\ L
\end{bmatrix}
Z
\begin{bmatrix}
Y \\ L
\end{bmatrix}^\top \cr
\end{pmat}
\begin{pmat}[{}]
I \cr\- A^\top \cr B^\top \cr
\end{pmat} > 0,
\end{equation*}
where $Z := (Y - \frac{1}{\gamma^2}I)^{-1}$. We thus have the following theorem. 

\begin{figure*}[b!]
\normalsize
\vspace*{4pt}
\hrulefill
\begin{equation}
\begin{bmatrix}
    Y- \beta I & 0 & 0 & 0 & C_{Y,L}^\top \\
    0 & 0 & 0 & Y & 0  \\
    0 & 0 & 0 & L & 0 \\
    0 & Y & L^\top & Y-\frac{1}{\gamma^2}I & 0 \\
    C_{Y,L} & 0 & 0 & 0 & I
    \end{bmatrix} - \alpha \begin{bmatrix}
    I & X_+ \\ 0 & -X_- \\ 0 & -U_- \\ 0 & 0 \\ 0 & 0
    \end{bmatrix}
    \begin{bmatrix}
    \Phi_{11} & \Phi_{12} \\
    \Phi_{12}^\top & \Phi_{22}
    \end{bmatrix}
    \begin{bmatrix}
    I & X_+ \\ 0 & -X_- \\ 0 & -U_- \\ 0 & 0 \\ 0 & 0
    \end{bmatrix}^\top \geq 0, \quad 
    Y -\frac{1}{\gamma^2} I > 0.
    \label{LMIHinf}\tag{$\mathcal{H}_\infty$}
\end{equation}
\end{figure*}

\begin{theorem}
\label{theoremHinf}
Assume that the generalized Slater condition \eqref{matSlater} holds for $N$ in \eqref{Nstab} and some $\bar{Z} \in \mathbb{R}^{(n+m)\times n}$. Then the data $(U_-,X)$ are informative for $\mathcal{H}_\infty$ control with performance $\gamma$ if and only if there exist matrices $Y = Y^\top > 0$ and $L$, and scalars $\alpha \geq 0$ and $\beta > 0$ satisfying \eqref{LMIHinf}.

Moreover, if $Y$ and $L$ satisfy \eqref{LMIHinf} then $K := L Y^{-1}$ is such that $A+BK$ is stable and $\norm{G(z)}_{\mathcal{H}_\infty} < \gamma$ for all $(A,B) \in \Sigma$.
\end{theorem}

The proof of Theorem~\ref{theoremHinf} is based on Theorem~\ref{strictmatSlemma2}. It follows similar steps as the proof of Theorem~\ref{theoremH2}, and is therefore not reported here. 

\section{Examples}
\label{sectionsimulations}
In this section we illustrate our theoretical results by examples and numerical simulations. 

\subsection{Stabilization using bounds on the noise samples}
Consider an unstable system of the form \eqref{system1} with $A_s$ and $B_s$ given by 
$$
A_s = 
\begingroup 
\setlength\arraycolsep{2pt}
\begin{bmatrix}
0.850 &  -0.038  & -0.380 \\
    0.735  &  0.815  &  1.594 \\
   -0.664  &  0.697  & -0.064
\end{bmatrix}
\endgroup
, B_s = \begin{bmatrix}
1.431  &  0.705 \\
    1.620  & -1.129 \\
    0.913 &   0.369
\end{bmatrix}.
$$
In this example, we assume that the noise samples $w(t)$ are bounded in norm as $\norm{w(t)}_2^2 \leq \epsilon$ for all $t$. As explained in Section \ref{sectionproblem}, we can capture this prior knowledge using the noise model \eqref{asnoise} with $\Phi_{11} = T \epsilon I$, $\Phi_{12} = 0$ and $\Phi_{22} -I$. We pick a time horizon of $T = 20$ and draw the entries of the inputs and initial state randomly from a Gaussian distribution with zero mean and unit variance. The noise samples are drawn uniformly at random from the ball $\{ w \in \mathbb{R}^3 \mid \norm{w}^2_2 \leq \epsilon \}$. We aim at constructing stabilizing controllers from the input/state data for various values of $\epsilon$. In particular, we investigate six different noise levels: $\epsilon \in \{0.5, 1, 1.5, 2, 2.2, 2.4\}$. For each noise level, we generate $100$ data sets using the method described above. We check the generalized Slater condition \eqref{matSlater} by verifying that $N$ in \eqref{Nstab} has $3$ positive eigenvalues; this turns out to be true for all $600$ data sets. For each noise level, we record the percentage of data sets from which a stabilizing controller was found for $(A_s,B_s)$ using the formulation \eqref{LMIstab}. We display the results in the following table.

\begin{center}
\begin{tabular}{ |c|c|c|c|c|c| }
 \hline
$\epsilon = 0.5$  & $\epsilon =1$ &  $\epsilon =1.5$ & $\epsilon =2$ & $\epsilon =2.2$ & $\epsilon =2.4$ \\
 \hline
$100 \%$ & $96 \%$ & $90 \%$ & $82 \%$ & $75 \%$ & $73 \%$ \\
 \hline
\end{tabular}
\end{center}

For $\epsilon = 0.5$ we find a stabilizing controller in all $100$ cases. When the noise level increases, the percentage of data sets for which the LMI \eqref{LMIstab} is feasible decreases. The interpretation is that by increasing the noise we enlarge the set of explaining systems $\Sigma$. It thus becomes harder to simultaneously stabilize the systems in $\Sigma$. Nonetheless, even for the larger noise level of $\epsilon = 2.4$ we find a stabilizing controller in $73$ out of the $100$ data sets.

\subsection{$\mathcal{H}_2$ control of a fighter aircraft}
We consider a state-space model of a fighter aircraft \cite[Ex. 10.1.2]{Skelton1997}. In particular, we discretize the model of \cite{Skelton1997} using a sampling time of $0.01$, which results in the (unstable) system of the form \eqref{system1} with $A_s$ and $B_s$ given by
\begin{align*}
&\begin{bmatrix}
    1.000 &  -0.374 &  -0.190  & -0.321  &  0.056  & -0.026 \\
    0.000  &  0.982  &  0.010  & -0.000 &  -0.003  &  0.001 \\
    0.000  &  0.115  &  0.975  & -0.000  & -0.269 &   0.191 \\
    0.000  &  0.001   &  0.010  &  1.000  & -0.001  &  0.001 \\
         0.000  &   0.000      &  0.000    &     0.000  &  0.741      &   0.000 \\
         0.000  &       0.000    &     0.000  & 0.000    &     0.000   & 0.741 \\
\end{bmatrix}, \\
&\begin{bmatrix}
0.007 & 0.000 & -0.043 & 0.000 & 0.259 & 0.000 \\
-0.003 & 0.000 & 0.030 & 0.000 & 0.000 & 0.259
\end{bmatrix}^\top,
\end{align*}
respectively. We consider the performance output as in \eqref{CD} with 
$$
C = \begin{bmatrix}
0 & 0 & 0 & 0 & 0 & 1
\end{bmatrix}
$$
and $D = 0$. First, we look for the smallest $\gamma$ such that \eqref{H2iff} is feasible for $(A_s,B_s)$. This minimum value of $\gamma$ is $1.000$ and can be regarded as a benchmark: no data-driven method can perform better than the model-based solution using full knowledge of $(A_s,B_s)$.

Of course, our goal is not to use the knowledge of $(A_s,B_s)$ but to seek a data-driven solution instead. Therefore, we collect $T = 750$ input and state samples of \eqref{system1}. The entries of the inputs and initial state were drawn randomly from a Gaussian distribution with zero mean and unit variance. Also the noise samples were drawn randomly from a Gaussian distribution, with zero mean and variance $\sigma^2$ with $\sigma = 0.005$. In this example, we assume knowledge of a bound on the energy of the noise as 
\begin{equation}
\label{boundWH2}
W_- W_-^\top \leq 1.35 T \sigma^2 I.
\end{equation}
We verified that this bound is satisfied for the generated noise sequence. In addition, we verified that the matrix $N$ in \eqref{Nstab} has $6$ positive eigenvalues, thus the generalized Slater condition \eqref{matSlater} holds. 

Next, we want to compute an $\mathcal{H}_2$ controller for the unknown system using the generated data. We do so by minimizing $\gamma$ subject to \eqref{LMIH2}. This is a semidefinite program that we solve in Matlab, using Yalmip \cite{Lofberg} with Mosek as an LMI solver. The obtained controller $K$ stabilizes the original system $(A_s,B_s)$. In addition, the system, in feedback with $K$, has an $\mathcal{H}_2$ norm of $\gamma_s$ where $\gamma^2_s = 1.007$. We note that this is almost identical to the smallest possible $\mathcal{H}_2$ norm of $1.000$.

Subsequently, we repeat the above experiment using only a \emph{part} of our data set. In particular, we compute an $\mathcal{H}_2$ controller via the semidefinite program as before, using only the first $i$ samples of $X_+, X_-$ and $U_-$ for $i = 50,100,\dots,750$. We display the results in Figure \ref{fig:plotH2}.
\begin{figure}[ht]
    \centering
    \includegraphics[trim=30 0 50 10,clip,width=0.48\textwidth]{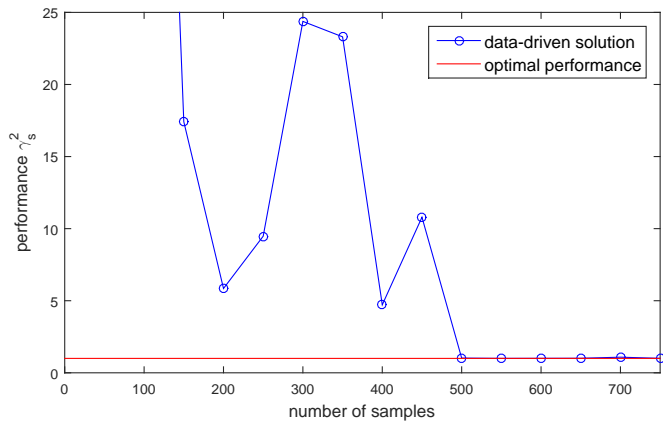}
    \caption{Achieved $\mathcal{H}_2$ performance of the true system in feedback with a data-based controller (blue) and the optimal (model-based) performance of the true system (red).}
    \label{fig:plotH2}
\end{figure}

In each of the cases a stabilizing controller was found from data. However, the performance of these controllers when applied to the true system varies, and is quite poor for $i < 500$. Starting from $i = 500$ and onward, the performance is close to the optimal performance of the true system. 

Next, we investigate what happens when we increase the variance $\sigma^2$ of the noise. First, we take $\sigma = 0.05$. We again generate $750$ data samples, and assume the same bound on the noise. The $\mathcal{H}_2$ controller achieves a performance of $\gamma^2_s = 1.146$ when interconnected to the true system. Increasing the variance of the noise has the effect that the set $\Sigma$ of explaining systems becomes larger. As such, it is more difficult to control all systems in $\Sigma$ resulting in a slightly larger $\gamma_s$. This behavior becomes even more apparent when increasing the variance of the noise to $\sigma = 0.5$. In this case we obtain a controller that yields a performance of $\gamma_s^2 = 3.579$. Increasing $\sigma$ even more to $\sigma = 1$ results in infeasibility of the LMI's \eqref{LMIH2} for any $\gamma$; the set of explaining systems has become too large for a quadratically stabilizing controller to exist. 

We remark that the size of the set $\Sigma$ does not only depend on the variance of the noise, but also on the available bound on the noise. Throughout this example, we have used the bound \eqref{boundWH2}. However, if we reconsider the case of $\sigma = 0.5$ with the tighter bound $W_-W_-^\top \leq 1.22 T \sigma^2 I$ (which is also satisfied in this example) we obtain a controller with better performance $\gamma_s^2 = 2.706$. This illustrates the simple fact that data-driven controllers not only depend on the particular design strategy, but also on the \emph{prior knowledge} on the noise. 

We conclude the example with a remark on the dimension of the variables involved in the formulation \eqref{LMIH2}. The symmetric matrices $Y$ and $Z$ both have $21$ free variables. The matrix $L$ contains $12$ variables, and $\alpha$ and $\beta$ are both scalar variables. Thus, the total number of variables is $56$. The size of the largest LMI in \eqref{LMIH2} is $21 \times 21$. We emphasize that our approach is based directly on the set $\Sigma$ of open-loop systems and avoids the parameterization of closed-loop systems, as employed in \cite{DePersis2020,Berberich2019c}. Such parameterizations involve decision variables of dimension $T \times n$, which would result in at least $4500$ variables in this example.

\subsection{Comparison with related results}
As we have mentioned before, one of the advantages of our approach compared to existing work is that our LMI condition provides necessary and sufficient conditions for data-driven quadratic stabilization. The purpose of this example is to demonstrate that our results can thus lead to stabilizing controllers even in situations where the results from \cite{DePersis2020,Berberich2019c} cannot. 

Consider the system \eqref{system1}, where $A_s = 1$ and $B_s = 1$. Suppose that $T = 3$ and the noise matrix is given by
$$
W_- = \begin{bmatrix}
\frac{1}{2} & \frac{1}{2} & \frac{1}{2}
\end{bmatrix}.
$$
We collect the data samples 
\begin{align*}
X &= \begin{bmatrix}
0 & 0 & 1 & 0
\end{bmatrix} \\
U_- &= \begin{bmatrix}
-\frac{1}{2} & \frac{1}{2} & -\frac{3}{2}
\end{bmatrix}.
\end{align*}
Throughout the example, we assume that we have access to the noise bound $W_- W_-^\top \leq 1$. Note that this bound is indeed satisfied, and that it can be captured using Assumption \ref{assumption1} by the choices $\Phi_{11} = 1$, $\Phi_{12} = 0$ and $\Phi_{22} = -I$. We also note that by Remark \ref{remarkBerb}, this noise bound can be captured by \cite[Asm. 3]{Berberich2019c} with the choices $Q_w = -1$, $S_w = 0$ and $R_w = I$. Finally, we note that the noise bound can be captured by \cite[Asm. 5]{DePersis2020} with the choice $\gamma = 1$. As such, we can compare the design methods reported in Theorem \ref{theoremstab} of this paper with the approaches in \cite[Cor. 6, Rem. 7]{Berberich2019c} and \cite[Thm. 6]{DePersis2020}\footnote{We note that \cite{DePersis2020} studies stabilization in the setting that $\bmw$ represents a bounded nonlinearity. The interpretation of $\bmw$, however, is not important for this comparison. In fact, our results (as well as those in \cite{Berberich2019c}) are applicable to general bounded disturbances, hence also to disturbances resulting from bounded nonlinearities.}.

We start by applying Theorem \ref{theoremstab} to the data in this example. Note that $X_- = \begin{bmatrix}
0 & 0 & 1
\end{bmatrix}$ and $X_+ = \begin{bmatrix}
0 & 1 & 0
\end{bmatrix}$. 
It can be easily verified that $(P,L,\alpha,\beta) = (0.9,-1.35,1.1,0.18)$ is a solution to \eqref{LMIstab}. In addition, the Slater condition is satisfied in this example. As such, we conclude by Theorem \ref{theoremstab} that the controller $K = L/P = -1.5$ is stabilizing for all $(A,B) \in \Sigma$. In particular, we see that the true closed-loop system matrix $-0.5$ is stable. 

We will now investigate the design method of \cite[Thm. 6]{DePersis2020}. This approach involves finding a matrix $Q$ and a scalar $\alpha > 0$ such that $X_- Q$ is symmetric and 
\begin{align}
\label{DePersis1}
\begin{bmatrix}
X_- Q - \alpha X_+ X_+^\top & X_+ Q \\ Q^\top X_+^\top & X_- Q
\end{bmatrix} >0, \:\: \begin{bmatrix}
I & Q \\ Q^\top & X_- Q
\end{bmatrix} > 0, \\
\label{DePersis2}
\frac{\alpha^2}{4+2\alpha} > \gamma,
\end{align} 
where we recall that $\gamma = 1$ in this example. We will now show that \eqref{DePersis1}, \eqref{DePersis2} do not have a solution $(Q,\alpha)$, and thus, the design procedure from \cite[Thm. 6]{DePersis2020} cannot find a controller that is guaranteed to stabilize the true system. To see this, note that \eqref{DePersis2} and $\alpha > 0$ imply $\alpha > 1+\sqrt{5}$. We write $Q = \begin{bmatrix}
q_1 & q_2 & q_3
\end{bmatrix}^\top$ and note that by the upper left block of the first matrix in \eqref{DePersis1}, we have $q_3 > \alpha$, thus $q_3 > 1$. Now, by taking the Schur complement of the second matrix in \eqref{DePersis1} with respect to the upper left block, we obtain $q_3 > q_1^2 + q_2^2 + q_3^2$. However, this inequality cannot be satisfied since $q_3 > 1$. As such, we conclude that \eqref{DePersis1}, \eqref{DePersis2} do not have a solution. 

Next, we turn our attention to the design procedure of \cite[Rem. 7]{Berberich2019c}. For $S_w = 0$, this procedure boils down to finding a solution $(\mathcal{Y},M)$ to 
\begin{align}
\label{Berberich1}
\begin{bmatrix}
-\mathcal{Y} & 0 & M^\top X_+^\top & M^\top \\
0 & Q_w & 1 & 0 \\
X_+ M & 1 & -\mathcal{Y} & 0 \\ 
M & 0 & 0 & -R_w^{-1}
\end{bmatrix} &< 0, \\
\label{Berberich2}
X_- M &= \mathcal{Y},
\end{align}
where we recall that $Q_w = -1$ and $R_w = I$ in this example. We will now show that \eqref{Berberich1}, \eqref{Berberich2} do not have a solution $(\mathcal{Y},M)$. To see this, note that \eqref{Berberich2} implies $M = \begin{bmatrix}
m_1 & m_2 & \mathcal{Y}
\end{bmatrix}^\top$ with $m_1,m_2 \in \mathbb{R}$. The negative definiteness of the submatrix of \eqref{Berberich1} consisting of the second and third row and column imply $\mathcal{Y} > 1$. Moreover, by inspection of the first and fourth row and column block of \eqref{Berberich1} we see that $-\mathcal{Y} + M^\top M < 0$ and thus, $-\mathcal{Y} + m_1^2+m_2^2+\mathcal{Y}^2 < 0$. This inequality, however, cannot be satisfied as $\mathcal{Y} > 1$. As such, we see that \eqref{Berberich1} and \eqref{Berberich2} do not have a solution $(\mathcal{Y},M)$. 

We conclude that Theorem \ref{theoremstab} can be successfully applied even in situations in which the design procedures of \cite{Berberich2019c,DePersis2020} do not lead to controllers that are guaranteed to stabilize the true system.

\section{Discussion and conclusions}
\label{sectionconclusions}

We have studied the problem of obtaining feedback controllers from noisy data. The essence of our approach has been to formulate data-driven control as the problem of determining when one quadratic matrix inequality implies another one. To get a grip on this fundamental question, we have generalized the classical S-lemma \cite{Yakubovich1977} to matrix variables. The implication involving quadratic matrix inequalities is thereby \emph{equivalent} to a linear matrix inequality in a scalar variable. We have established several versions of the matrix S-lemma, for both strict and non-strict inequalities. These matrix S-lemmas are interesting in their own right, and generalize existing S-lemmas \cite{Polik2007} as well as a theorem involving quadratic matrix inequalities \cite{Luo2004}. 

We have followed up by applying our matrix S-lemma to data-driven control. In particular, we have given necessary and sufficient conditions under which stabilizing, $\mathcal{H}_2$, and $\mathcal{H}_\infty$ controllers can be obtained from noisy data. Our control design revolves around data-guided linear matrix inequalities, which can be solved efficiently using modern LMI solvers. In addition to being non-conservative, an attractive feature of our design procedure is that decision variables are \emph{independent} of the time horizon of the experiment. 

So far, we have only applied the matrix S-lemma involving a strict inequality (Thms. \ref{strictmatSlemma}, \ref{strictmatSlemma2}) to data-driven control. However, we are convinced that also the matrix S-lemma with \emph{non-strict} inequalities (Thm. \ref{matSlemma}) will find applications, for example, in the verification of dissipativity properties from data \cite{Koch2020}. 

The noise model that we have employed is flexible, and can describe, e.g., constant disturbances, energy bounded noise and norm bounds on noise samples. If one is only interested in the latter, however, we expect that more specific control techniques are possible. In fact, analogous to \eqref{ineqAB}, we can write the inequality $w(t)^\top w(t) \leq \epsilon$ as 
\begin{equation*}
    \begin{bmatrix}
    I \\ A^\top \\ B^\top 
    \end{bmatrix}^\top 
    \begingroup 
\setlength\arraycolsep{2pt}
   \underbrace{\begin{bmatrix}
    I & x(t+1) \\ 0 & -x(t) \\ 0 & -u(t)
    \end{bmatrix}
    \begin{bmatrix}
    \epsilon I & 0 \\
    0 & -I
    \end{bmatrix}
    \begin{bmatrix}
    I & x(t+1) \\ 0 & -x(t) \\ 0 & -u(t)
    \end{bmatrix}^\top}_{:= N_t}
    \endgroup
    \begin{bmatrix}
    I \\ A^\top \\ B^\top 
    \end{bmatrix} \geq 0.
\end{equation*}
In the spirit of the S-procedure, one could thus design a stabilizing controller by computing\footnote{This procedure is likely to be conservative, however, since the classical S-lemma is in general conservative for more than two quadratic functions \cite{Polik2007}.} matrices $P = P^\top > 0$ and $K$, and \emph{multiple} non-negative scalars $\alpha_0,\alpha_1,\dots,\alpha_{T-1}$ such that
$$
M-\sum_{t = 0}^{T-1} \alpha_t N_t > 0,
$$
with $M$ given by \eqref{Mstab}. We will consider norm bounded noise samples in more detail in future work. 

Yet another idea for future work is to extend the current results for state-feedback design to data-driven dynamic output feedback design. Specifically, it would be interesting to see whether the matrix S-lemmas can be applied to obtain dynamic output feedback controllers from a finite set of noisy \emph{input/output} samples.

\bibliographystyle{IEEEtran}
\bibliography{references}

\begin{thebibliography}{10}
\providecommand{\url}[1]{#1}
\csname url@samestyle\endcsname
\providecommand{\newblock}{\relax}
\providecommand{\bibinfo}[2]{#2}
\providecommand{\BIBentrySTDinterwordspacing}{\spaceskip=0pt\relax}
\providecommand{\BIBentryALTinterwordstretchfactor}{4}
\providecommand{\BIBentryALTinterwordspacing}{\spaceskip=\fontdimen2\font plus
\BIBentryALTinterwordstretchfactor\fontdimen3\font minus
  \fontdimen4\font\relax}
\providecommand{\BIBforeignlanguage}[2]{{%
\expandafter\ifx\csname l@#1\endcsname\relax
\typeout{** WARNING: IEEEtran.bst: No hyphenation pattern has been}%
\typeout{** loaded for the language `#1'. Using the pattern for}%
\typeout{** the default language instead.}%
\else
\language=\csname l@#1\endcsname
\fi
#2}}
\providecommand{\BIBdecl}{\relax}
\BIBdecl

\bibitem{Yakubovich1977}
V.~A. Yakubovich, ``S-procedure in nonlinear control theory,'' \emph{Vestnik
  Leningrad University Mathematics}, vol.~4, pp. 73--93, 1977.

\bibitem{Polik2007}
I.~P\'olik and T.~Terlaky, ``A survey of the {S}-lemma,'' \emph{SIAM Review},
  vol.~49, no.~3, pp. 371--418, 2007.

\bibitem{vanWaarde2020}
H.~J. {van Waarde}, J.~{Eising}, H.~L. {Trentelman}, and M.~K. {Camlibel},
  ``Data informativity: a new perspective on data-driven analysis and
  control,'' \emph{IEEE Transactions on Automatic Control}, pp. 1--1, 2020.

\bibitem{Hou2013}
Z.~Hou and Z.~Wang, ``From model-based control to data-driven control: Survey,
  classification and perspective,'' \emph{Information Sciences}, vol. 235, pp.
  3--35, 2013.

\bibitem{Skelton1994}
R.~E. {Skelton} and {G. Shi}, ``The data-based {LQG} control problem,'' in
  \emph{Proceedings of the IEEE Conference on Decision and Control}, Dec 1994,
  pp. 1447--1452.

\bibitem{Furuta1995}
K.~Furuta and M.~Wongsaisuwan, ``Discrete-time {LQG} dynamic controller design
  using plant {Markov} parameters,'' \emph{Automatica}, vol.~31, no.~9, pp.
  1317--1324, 1995.

\bibitem{Shi1998}
G.~Shi and R.~Skelton, ``Markov data-based {LQG} control,'' \emph{ASME Journal
  of Dynamical Systems, Measurement, and Control}, vol. 122, no.~3, pp.
  551--559, 1998.

\bibitem{Favoreel1999b}
W.~{Favoreel}, B.~{De Moor}, P.~{Van Overschee}, and M.~{Gevers}, ``Model-free
  subspace-based {LQG}-design,'' in \emph{Proceedings of the American Control
  Conference}, vol.~5, June 1999, pp. 3372--3376.

\bibitem{Aangenent2005}
W.~{Aangenent}, D.~{Kostic}, B.~{de Jager}, R.~{van de Molengraft}, and
  M.~{Steinbuch}, ``Data-based optimal control,'' in \emph{Proceedings of the
  American Control Conference}, June 2005, pp. 1460--1465.

\bibitem{Pang2018}
B.~{Pang}, T.~{Bian}, and Z.~{Jiang}, ``Data-driven finite-horizon optimal
  control for linear time-varying discrete-time systems,'' in \emph{Proceedings
  of the IEEE Conference on Decision and Control}, Dec 2018, pp. 861--866.

\bibitem{daSilva2019}
G.~R. {Gon\c{c}alves da Silva}, A.~S. Bazanella, C.~Lorenzi, and
  L.~Campestrini, ``Data-driven {LQR} control design,'' \emph{IEEE Control
  Systems Letters}, vol.~3, no.~1, pp. 180--185, 2019.

\bibitem{Baggio2019}
G.~{Baggio}, V.~{Katewa}, and F.~{Pasqualetti}, ``Data-driven minimum-energy
  controls for linear systems,'' \emph{IEEE Control Systems Letters}, vol.~3,
  no.~3, pp. 589--594, July 2019.

\bibitem{Mukherjee2018}
S.~{Mukherjee}, H.~{Bai}, and A.~{Chakrabortty}, ``On model-free reinforcement
  learning of reduced-order optimal control for singularly perturbed systems,''
  in \emph{IEEE Conference on Decision and Control}, Dec 2018, pp. 5288--5293.

\bibitem{Alemzadeh2019}
S.~{Alemzadeh} and M.~{Mesbahi}, ``Distributed {Q}-learning for dynamically
  decoupled systems,'' in \emph{Proceedings of the American Control
  Conference}, July 2019, pp. 772--777.

\bibitem{Dean2019}
S.~Dean, H.~Mania, N.~Matni, B.~Recht, and S.~Tu, ``On the sample complexity of
  the linear quadratic regulator,'' \emph{Foundations of Computational
  Mathematics}, Aug 2019.

\bibitem{vanWaarde2020c}
H.~J. van Waarde and M.~Mesbahi, ``Data-driven parameterizations of suboptimal
  {LQR} and {H2} controllers,'' \emph{https://arxiv.org/abs/1912.07671}, 2020.

\bibitem{Keel2008}
L.~H. {Keel} and S.~P. {Bhattacharyya}, ``Controller synthesis free of
  analytical models: Three term controllers,'' \emph{IEEE Transactions on
  Automatic Control}, vol.~53, no.~6, pp. 1353--1369, July 2008.

\bibitem{Fliess2013}
M.~Fliess and C.~Join, ``Model-free control,'' \emph{International Journal of
  Control}, vol.~86, no.~12, pp. 2228--2252, 2013.

\bibitem{Favoreel1999}
W.~Favoreel, B.~{De Moor}, and M.~Gevers, ``{SPC}: Subspace predictive
  control,'' \emph{IFAC Proceedings Volumes}, vol.~32, no.~2, pp. 4004--4009,
  1999.

\bibitem{Salvador2018}
J.~R. Salvador, D.~{Mu\~noz de la Pe\~na}, T.~Alamo, and A.~Bemporad,
  ``Data-based predictive control via direct weight optimization,''
  \emph{IFAC-PapersOnLine}, vol.~51, no.~20, pp. 356--361, 2018.

\bibitem{Huang2019}
L.~Huang, J.~Coulson, J.~Lygeros, and F.~D\"orfler, ``Data-enabled predictive
  control for grid-connected power converters,''
  \emph{https://arxiv.org/abs/1903.07339}, 2019.

\bibitem{Alpago2020}
D.~Alpago, F.~D\"orfler, and J.~Lygeros, ``An extended {Kalman} filter for
  data-enabled predictive control,'' \emph{https://arxiv.org/abs/2003.08269},
  2020.

\bibitem{Hewing2020}
L.~Hewing, K.~P. Wabersich, M.~Menner, and M.~N. Zeilinger, ``Learning-based
  model predictive control: Toward safe learning in control,'' \emph{Annual
  Review of Control, Robotics, and Autonomous Systems}, vol.~3, no.~1, pp.
  269--296, 2020.

\bibitem{Tabuada2017}
P.~{Tabuada}, W.~{Ma}, J.~{Grizzle}, and A.~D. {Ames}, ``Data-driven control
  for feedback linearizable single-input systems,'' in \emph{Proceedings of the
  IEEE Conference on Decision and Control}, 2017, pp. 6265--6270.

\bibitem{Dai2018}
T.~{Dai} and M.~{Sznaier}, ``A moments based approach to designing {MIMO} data
  driven controllers for switched systems,'' in \emph{Proceedings of the IEEE
  Conference on Decision and Control}, 2018, pp. 5652--5657.

\bibitem{Tabuada2020}
P.~Tabuada and L.~Fraile, ``Data-driven stabilization of {SISO} feedback
  linearizable systems,'' \emph{https://arxiv.org/abs/2003.14240}, 2020.

\bibitem{Bisoffi2020}
A.~Bisoffi, C.~{De Persis}, and P.~Tesi, ``Data-based stabilization of unknown
  bilinear systems with guaranteed basin of attraction,''
  \emph{https://arxiv.org/abs/2004.11630}, 2020.

\bibitem{Guo2020}
M.~Guo, C.~{De Persis}, and P.~Tesi, ``Learning control for polynomial systems
  using sum of squares,'' \emph{https://arxiv.org/abs/2004.00850}, 2020.

\bibitem{Bradtke1993}
S.~J. Bradtke, ``Reinforcement learning applied to linear quadratic
  regulation,'' in \emph{Advances in Neural Information Processing Systems},
  1993, pp. 295--302.

\bibitem{Hjalmarsson1998}
H.~{Hjalmarsson}, M.~{Gevers}, S.~{Gunnarsson}, and O.~{Lequin}, ``Iterative
  feedback tuning: theory and applications,'' \emph{IEEE Control Systems
  Magazine}, vol.~18, no.~4, pp. 26--41, Aug 1998.

\bibitem{Campi2002}
M.~C. Campi, A.~Lecchini, and S.~M. Savaresi, ``Virtual reference feedback
  tuning: a direct method for the design of feedback controllers,''
  \emph{Automatica}, vol.~38, no.~8, pp. 1337--1346, 2002.

\bibitem{Willems2005}
J.~C. Willems, P.~Rapisarda, I.~Markovsky, and B.~L.~M. {De Moor}, ``A note on
  persistency of excitation,'' \emph{Systems \& Control Letters}, vol.~54,
  no.~4, pp. 325--329, 2005.

\bibitem{vanWaarde2020b}
H.~J. {van Waarde}, C.~{De Persis}, M.~K. {Camlibel}, and P.~{Tesi}, ``Willems'
  fundamental lemma for state-space systems and its extension to multiple
  datasets,'' \emph{IEEE Control Systems Letters}, vol.~4, no.~3, pp. 602--607,
  2020.

\bibitem{Markovsky2008}
I.~Markovsky and P.~Rapisarda, ``Data-driven simulation and control,''
  \emph{International Journal of Control}, vol.~81, no.~12, pp. 1946--1959,
  2008.

\bibitem{Maupong2017}
T.~M. Maupong and P.~Rapisarda, ``Data-driven control: A behavioral approach,''
  \emph{Systems \& Control Letters}, vol. 101, pp. 37--43, 2017.

\bibitem{Coulson2019}
J.~{Coulson}, J.~{Lygeros}, and F.~{D\"orfler}, ``Data-enabled predictive
  control: In the shallows of the {DeePC},'' in \emph{Proceedings of the
  European Control Conference}, June 2019, pp. 307--312.

\bibitem{DePersis2020}
C.~{De Persis} and P.~{Tesi}, ``Formulas for data-driven control:
  Stabilization, optimality, and robustness,'' \emph{IEEE Transactions on
  Automatic Control}, vol.~65, no.~3, pp. 909--924, 2020.

\bibitem{DePersis2020b}
C.~{De Persis} and P.~Tesi, ``Low-complexity learning of linear quadratic
  regulators from noisy data,'' \emph{https://arxiv.org/abs/2005.01082}, 2020.

\bibitem{Berberich2019c}
J.~{Berberich}, A.~{Koch}, C.~W. {Scherer}, and F.~{Allg\"ower}, ``Robust
  data-driven state-feedback design,'' in \emph{Proceedings of the American
  Control Conference}, 2020, pp. 1532--1538.

\bibitem{Allibhoy2020}
A.~Allibhoy and J.~Cort\'es, ``Data-based receding horizon control of linear
  network systems,'' \emph{https://arxiv.org/abs/2003.09813}, 2020.

\bibitem{Baggio2020}
G.~Baggio, D.~S. Bassett, and F.~Pasqualetti, ``Data-driven control of complex
  networks,'' \emph{https://arxiv.org/abs/2003.12189}, 2020.

\bibitem{Monshizadeh2020}
N.~{Monshizadeh}, ``Amidst data-driven model reduction and control,''
  \emph{IEEE Control Systems Letters}, vol.~4, no.~4, pp. 833--838, 2020.

\bibitem{Boyd1994}
S.~Boyd, L.~E. Ghaoui, E.~Feron, and V.~Balakrishnan, \emph{Linear Matrix
  Inequalities in System and Control Theory}, ser. Studies in Applied
  Mathematics.\hskip 1em plus 0.5em minus 0.4em\relax Society for Industrial
  and Applied Mathematics, 1994.

\bibitem{Iwasaki2000}
T.~Iwasaki, G.~Meinsma, and M.~Fu, ``Generalized {S}-procedure and finite
  frequency {KYP} lemma,'' \emph{Mathematical Problems in Engineering}, vol.~6,
  pp. 305--320, Jan. 2000.

\bibitem{Iwasaki2005}
T.~{Iwasaki} and S.~{Hara}, ``Generalized {KYP} lemma: unified frequency domain
  inequalities with design applications,'' \emph{IEEE Transactions on Automatic
  Control}, vol.~50, no.~1, pp. 41--59, 2005.

\bibitem{Luo2004}
Z.-Q. Luo, J.~F. Sturm, and S.~Zhang, ``Multivariate nonnegative quadratic
  mappings,'' \emph{SIAM Journal on Optimization}, vol.~14, no.~4, pp.
  1140--1162, 2004.

\bibitem{Scherer1997}
C.~W. {Scherer}, ``A full block {S-procedure} with applications,'' in
  \emph{Proceedings of the IEEE Conference on Decision and Control}, vol.~3,
  1997, pp. 2602--2607.

\bibitem{Iwasaki1997}
T.~{Iwasaki}, ``Robust stability analysis with quadratic separator: parametric
  time-varying uncertainty case,'' in \emph{Proceedings of the IEEE Conference
  on Decision and Control}, vol.~5, 1997, pp. 4880--4885.

\bibitem{Scherer2001}
C.~W. Scherer, ``{LPV} control and full block multipliers,'' \emph{Automatica},
  vol.~37, no.~3, pp. 361--375, 2001.

\bibitem{Iwasaki2001}
T.~{Iwasaki} and G.~{Shibata}, ``{LPV} system analysis via quadratic separator
  for uncertain implicit systems,'' \emph{IEEE Transactions on Automatic
  Control}, vol.~46, no.~8, pp. 1195--1208, 2001.

\bibitem{Scherer2005}
C.~W. Scherer, ``Relaxations for robust linear matrix inequality problems with
  verifications for exactness,'' \emph{SIAM Journal on Matrix Analysis and
  Applications}, vol.~27, no.~2, pp. 365--395, 2005.

\bibitem{Scherer2006}
------, ``{LMI} relaxations in robust control,'' \emph{European Journal of
  Control}, vol.~12, no.~1, pp. 3--29, 2006.

\bibitem{Scherer2006b}
C.~W. Scherer and C.~W.~J. Hol, ``Matrix sum-of-squares relaxations for robust
  semi-definite programs,'' \emph{Mathematical Programming}, vol. 107, pp.
  189--211, 2006.

\bibitem{Umenberger2019}
J.~Umenberger, M.~Ferizbegovic, T.~B. Sch\"{o}n, and H.~Hjalmarsson, ``Robust
  exploration in linear quadratic reinforcement learning,'' in \emph{Advances
  in Neural Information Processing Systems 32}.\hskip 1em plus 0.5em minus
  0.4em\relax Curran Associates, Inc., 2019, pp. 15\,336--15\,346.

\bibitem{Ferizbegovic2020}
M.~{Ferizbegovic}, J.~{Umenberger}, H.~{Hjalmarsson}, and T.~B. {Sch\"on},
  ``{Learning Robust LQ-Controllers Using Application Oriented Exploration},''
  \emph{IEEE Control Systems Letters}, vol.~4, no.~1, pp. 19--24, 2020.

\bibitem{Iannelli2020}
A.~{Iannelli} and R.~S. {Smith}, ``A multiobjective {LQR} synthesis approach to
  dual control for uncertain plants,'' \emph{IEEE Control Systems Letters},
  vol.~4, no.~4, pp. 952--957, 2020.

\bibitem{Zi-zong2010}
Y.~Zi-zong and G.~Jin-hai, ``Some equivalent results with {Yakubovich's
  S-Lemma},'' \emph{SIAM Journal on Control and Optimization}, vol.~48, no.~7,
  pp. 4474--4480, 2010.

\bibitem{Koch2020}
A.~Koch, J.~Berberich, and F.~Allg\"ower, ``Verifying dissipativity properties
  from noise-corrupted input-state data,''
  \emph{https://arxiv.org/pdf/2004.07270.pdf}, 2020.

\bibitem{Skelton1997}
R.~Skelton, T.~Iwasaki, and D.~Grigoriadis, \emph{A Unified Algebraic Approach
  To Control Design}.\hskip 1em plus 0.5em minus 0.4em\relax Taylor \& Francis,
  1997.

\bibitem{Lofberg}
J.~{Lofberg}, ``{YALMIP}: a toolbox for modeling and optimization in
  {MATLAB},'' in \emph{IEEE International Conference on Robotics and
  Automation}, 2004, pp. 284--289.

\end{thebibliography}

\end{document}